\documentclass[10pt,a4paper]{article}
\usepackage{amsmath,amssymb,amsfonts,amsthm}
\usepackage{graphicx}
\usepackage{float,graphicx,color}
\usepackage[all]{xy}

\newcommand{\rmd}{\mathrm{d}}

\newcommand{\rmC}{\mathrm{C}}
\newcommand{\rmD}{\mathrm{D}}
\newcommand{\rmH}{\mathrm{H}}
\newcommand{\rmL}{\mathrm{L}}

\newcommand{\rmP}{\mathrm{P}}

\newcommand{\sfd}{\mathsf{d}}

\newcommand{\bbR}{\mathbb{R}}

\newcommand{\bbV}{\mathbb{V}}
\newcommand{\bbW}{\mathbb{W}}
\newcommand{\bbZ}{\mathbb{Z}}

\newcommand{\frp}{\mathfrak{p}}

\newcommand{\calC}{\mathcal{C}}

\newcommand{\calF}{\mathcal{F}}

\newcommand{\calL}{\mathcal{L}}

\newcommand{\calR}{\mathcal{R}}
\newcommand{\calS}{\mathcal{S}}
\newcommand{\calT}{\mathcal{T}}

\newcommand{\id}{\mathrm{id}}

\DeclareMathOperator{\Div}{div}
\renewcommand{\div}{\Div}

\DeclareMathOperator{\grad}{grad}
\DeclareMathOperator{\curl}{curl}

\DeclareMathOperator{\vect}{vect}
\DeclareMathOperator{\myspan}{span}

\newcommand{\ts}{\textstyle}

\newcommand{\bs}{{\scriptscriptstyle \bullet}}

\newcommand{\whitney}{\mathfrak{W}}

\newcommand{\diff}{\mathrm{d}}
\newcommand{\rest}{\mathrm{r}}
\newcommand{\eval}{\mathrm{e}}
\newcommand{\const}{\mathrm{c}}

\DeclareMathOperator{\orient}{\mathrm{o}}
\DeclareMathOperator{\eder}{\rmd}
\DeclareMathOperator{\pull}{\mathrm{pu}}
\DeclareMathOperator{\trace}{\mathrm{tr}}

\DeclareMathOperator{\ext}{\mathsf{ext}}

\renewcommand{\diff}{\mathsf{d}}
\renewcommand{\rest}{\mathsf{r}}
\renewcommand{\eval}{\mathsf{e}}

\renewcommand{\const}{\mathsf{c}}

\DeclareMathOperator{\poincare}{\frp}
\DeclareMathOperator{\koszul}{\kappa}
\DeclareMathOperator{\lie}{\calL}

\newcommand{\ula}[1]{\underleftarrow{#1}}

\newcommand{\subcell}{\unlhd} % {\vartriangleleft}
\newcommand{\subcellstrict}{\lhd}

\newcommand{\subcells}{\calS}

\newcommand{\ctr}{ \, \mathsf{L} \, }

\newcommand{\poly}{\rmP}

\newcommand{\alter}{\mathrm{Alt}}

\newcommand{\refi}{\calR}

\newcommand{\difo}{\Lambda}

\newcommand{\beq}{\begin{equation}}
\newcommand{\eeq}{\end{equation}}

\newcommand{\mapping}[4]{
\left\{
\begin{array}{rcl}
\displaystyle #1  &\to& #2 ,\\
\displaystyle #3  &\mapsto & #4.
\end{array} \right.
}

\definecolor{mygray}{gray}{0.75}
\definecolor{myyellow}{rgb}{0.9, 0.8,0.1}
\definecolor{myred}{rgb}{0.7, 0.15, 0.15}
\definecolor{myblue}{rgb}{0.15, 0.15, 0.7}

\newcounter{npoint}[section]

\newcommand{\case}[1]{
\bigskip
\noindent -- \emph{#1}
}

\newtheorem{theorem}{Theorem}[section]
\newtheorem{lemma}[theorem]{Lemma}

\newtheorem{proposition}[theorem]{Proposition}

\theoremstyle{definition}
\newtheorem{definition}{Definition}[section]

\theoremstyle{remark}
\newtheorem{remark}{Remark}[section]

\newcounter{quest}[exercise]

\newcounter{subquest}[quest]

\newcounter{pushlevel}[theorem]
\newcommand{\push}{
\stepcounter{pushlevel}
 \vspace{-1ex} \noindent \hspace{4.5ex} \begin{minipage}[t]{\textwidth*\real{0.99} - 4.5ex}
\mbox{}\hspace{-1ex}\rule[-1.5ex]{.1pt}{1.5ex}\rule{5.3ex}{.1pt}
\vspace{-.65ex}

}
\newcommand{\unpush}{
\end{minipage}
\noindent \mbox{}\hspace{4.2ex}\rule{.1pt}{1.5ex}\rule{5.3ex}{.1pt} \hspace{\stretch{1}}
\noindent
\addtocounter{pushlevel}{-1}

}

\newcounter{theopoint}[theorem]
\newcommand{\tpoint}{
\medskip
\stepcounter{theopoint}
\noindent (\roman{theopoint}) 
}

\theoremstyle{plain}

\theoremstyle{definition}

\theoremstyle{remark}

\newcounter{teopunkt}[teorem]

\newcounter{osp}[oppgave]

\usepackage{tikz}
\usepackage{pict2e}
\definecolor{lightorange}{rgb}{0.91, 0.84, 0.42}
\definecolor{deepcarrotorange}{rgb}{0.91, 0.41, 0.17}
\definecolor{deepsaffron}{rgb}{1.0, 0.6, 0.2}
\definecolor{gold}{rgb}{0.91, 0.84, 0.17}

\definecolor{red}{rgb}{0.8,0,0}
\definecolor{darkorange}{rgb}{1,0.4,0}
\definecolor{lightorange}{rgb}{1,0.6, 0}
\definecolor{yellow}{rgb}{1,0.8, 0}

\title{
Generalized Finite Element Systems\\
for smooth differential forms and Stokes' problem
}

\author{Snorre H. Christiansen\thanks{Department of Mathematics, University of Oslo, PO Box 1053 Blindern, NO 0316 Oslo, Norway. email: {\tt snorrec@math.uio.no.}}, Kaibo Hu\thanks{Department of Mathematics, University of Oslo, PO Box 1053 Blindern, NO 0316 Oslo, Norway.  email: {\tt kaibohu@math.uio.no.}}}

\date{}

\begin{document}

\maketitle

\begin{abstract}
We provide both a general framework for discretizing de Rham sequences of differential forms of high regularity, and some examples of finite element spaces that fit in the framework. The general framework is an extension of the previously introduced notion of Finite Element Systems, and the examples include conforming mixed finite elements for Stokes' equation. In dimension 2 we detail four low order finite element complexes and one infinite family of highorder finite element complexes. In dimension 3 we define one low order complex, which may be branched into Whitney forms at a chosen index. Stokes pairs with continuous or discontinuous pressure are provided in arbitrary dimension. The finite element spaces all consist of composite polynomials. The framework guarantees some nice properties of the spaces, in particular the existence of commuting interpolators. It also shows that some of the examples are minimal spaces.
\end{abstract}

\bigskip


\section*{Introduction}
This article is concerned with developing finite element complexes similar to those described in \cite{RavTho77}\cite{Ned80}\cite{Hip99}\cite{Chr07NM}\cite{ArnFalWin06}\cite{ChrRap16}, but with enhanced continuity properties. Finite element spaces should be compatible in a precise sense, which in general will depend on the partial differential equation one wants to solve and will reflect the functional framework one adopts for the analysis. One way of phrasing compatibility is that the discrete spaces should form a subcomplex of a certain Hilbert complex, whose norm reflects the desired continuity, and that they should be equipped with bounded projections that commute with the differential operators. The fields we seek to discretize here can all be interpreted as differential forms, and the relevant operators are instances of the exterior derivative. What we seek can then be called good discrete de Rham sequences. The finite elements described in the above cited works are only partially continuous: for vectorfields continuity holds only in either tangential or normal directions, at interfaces of the mesh. In this paper, full continuity is achieved. This is particularly relevant for the Stokes equation.

In many cases the bounded projections alluded to above, can be obtained by an averaging technique \cite{ChrWin08}\cite{ChrMunOwr11}, from the interpolators associated with degrees of freedom, defined on smooth differential forms. Since the averaging technique is defined so as to commute with the differential, we may concentrate on getting degrees of freedom that provide commuting interpolators. As it turns out, the existence of such degrees of freedom, on a finite element space, can be deduced from a few algebraic constraints, that have been clarified in a framework of Finite Element Systems (FES) introduced in \cite{Chr08M3AS} and further developed in \cite{Chr09AWM}\cite{ChrRap16}\cite{ChrGil16}. A precise notion of \emph{compatibility} guarantees that the so-called harmonic degrees of freedom are unisolvent and provide an interpolator that commutes. 

Recall Ciarlet's definition of a finite element (e.g. \cite{Cia91} \S 10), in terms of spaces equipped with degrees of freedom (DoF). The framework of FES gives DoFs a secondary role. Rather, compatibility is defined in terms of restrictions and differentials. On a compatible FES there will in general be \emph{many} choices of DoFs, for the \emph{same} spaces. The harmonic DoFs are a natural choice among these possibilities. As we will see DoFs seem most useful to describe low order elements, where there is not so much choice.

In this paper we provide both a generalization of the framework of FES that can handle higher order continuity of differential forms and some examples of new spaces that fit into the framework. The generalization essentially consists in allowing for other types of restriction operators than pullback of differential forms. These restriction operators reflect that higher continuity implies that more information about the fields should be available on interfaces in the mesh. The examples of FES we provide, are all composite finite elements on a simplicial mesh, that are piecewise polynomials with respect to a simplicial refinement. For the spaces of scalar functions, considered as $0$-forms, we use continuously differentiable composite elements, as introduced by Hsieh and Clough-Tocher and discussed further in \cite{DouDup79}\cite{Cia91}\cite{LaiSch07}\cite{Wal14}. The rest of the sequences, pertaining to differential $k$-forms for $k \geq 1$, appear to be new. These sequences end with conforming mixed finite elements for the Stokes equation: continuous vectorfields with either continuous or discontinuous divergence. 

Our results are quite closely related to those of \cite{FalNei13}, where finite element families for Stokes' equation are defined in 2D, for both $\rmH^1-\rmL^2$ conforming and $\rmH^1_{\div}-\rmH^1$ conforming settings, as part of de Rham sequences with high regularity. In \cite{Nei15} these results are extended to 3D. The spaces attached to triangles or tetrahedra consist of polynomials, including polynomial bubbles. In particular they are smooth functions. Their degrees of freedom for vectorfields include in particular all first order partial derivatives at vertices. Our local spaces of vectorfields, on the other hand, consist of composite polynomials which are not necessarily of class $\rmC^1$, and our degrees of freedom at vertices are just the vertex values and (for $\rmH^1_{\div}$) vertex values of the divergence, which is a particular combination of first order derivatives. We point out that our spaces come equipped with commuting interpolators. The lower order continuity/differentiability imposed at vertices (for instance), and the composite nature of our elements, seem important in this respect, from the point of view provided by FES, for the given continuity one wants to achieve. The commuting diagram that we obtain, as a consequence of compatibility, makes the proof of the inf-sup condition easier than the macro-element techniques introduced for Stokes in \cite{Ste84}. Or, at least, it provides an alternative type of proof.

We also mention a connection with \cite{GuzNei14a}\cite{GuzNei14b}. In two dimensions they construct a complex of spaces equipped with degrees of freedom that provide commuting interpolators, and such that the two last spaces form a Stokes pair. In dimension 3 they construct Stokes pairs equipped with degrees of freedom that make the interpolator commute. In both cases, the local spaces contain rational functions, where we have used composite polynomials for similar purposes. In dimension two their lowest order complex resolves a $\rmC^1$ element due to Zienkiewicz, whereas in our case we resolve the Clough-Tocher element. See Remarks \ref{rem:gn14a} and \ref{rem:gn14b} for further considerations.

%%%%%%%%%%%%%%%%%%%%%%%%

There is a vast literature on the construction of stable Stokes pairs. The most natural candidate seems to be the $\rmC^{0}\poly^{p} - \poly^{p-1}$ pair, where the velocity is discretised by Lagrange elements of degree $p$, and the pressure with discontinuous polynomials of degree $p-1$. This is called the Scott-Vogelius element \cite{ScoVog85}, which is easy to implement and leads to strong divergence-free discretisations ; actually $\div V_{h}\subseteq Q_{h}$, for velocity space $V_{h}$ and pressure space $Q_{h}$. However the surjectivity and inf-sup conditions are subtle. The divergence operator $\div : \rmC^0\poly^{p} \to \poly^{p-1}$ is onto when there are no ''singular vertices''. The definition  of singular vertex is clearcut in 2D ; in \cite{ScoVog85} it is shown that in 2D, when there is no singular vertex and $p\geq 4$, the inf-sup condition holds (with respect to $\rmH^1 - \rmL^2$ norms). In 3D, it remains open to define all singular vertices and edges, and find the minimal polynomial degree $p$, see \cite{Zha05}.

Instead of trying to identify singular vertices and edges, people also identify refinements of simplicial meshes, where the inf-sup condition holds: 

-- In 2D, on triangles with Clough-Tocher splits, stability of $\rmC^{0}\poly^{2} - \poly^{1}$ and $\rmC^{0}\poly^{3} - \poly^{2}$  approximations was shown in the thesis of Qin \cite{Qin94}, see also \cite{ArnQin92}. In 2D, the stability of quadratic velocity and linear pressure on crisscross triangulation can be found in \cite{ArnQin92}. On two dimensional Powell-Sabin splits, the $\rmC^0\poly^{1} - \poly^{0}$ pair is stable \cite{Zha08}.

-- The 3D case is more involved. When we subdivide a tetrahedra into four, by the Alfeld split that connects one internal point with the four vertices, the inf-sup condition was shown in \cite{Zha05}. The lowest degree in this case is $\rmC^0\poly^{4} - \poly^{3}$. On Powell-Sabin splits, $\rmC^0\poly^{2} - \poly^{1}$ is stable \cite{Zha11}.  

The main technique of proof in the above cases seems to be the macroelement technique of \cite{Ste84}. Here we rely instead on (often exact) sequences connected by cochain morphisms. In 2D we introduce sequences based on the Clough-Tocher $\rmC^1$ element, so that naturally we are led to the $\rmC^0\poly^{2} - \poly^{1}$ pair for Stokes, but not $\rmC^0\poly^{1} - \poly^{0}$.
%%%%%%%%%%%%%%%%%%

\bigskip

To be more specific on our contributions, we consider an $n$-dimensional domain $S$, say in the Euclidean space $ \bbR^n$. The space of alternating $k$-linear forms on $\bbR^n$ is denoted $\alter^k(\bbR^n)$. For $r \geq 0$ we denote by $\rmH^r\difo^k(S)$ the spaces of $k$-forms on $S$ with partial derivatives up to order $r$ in $\rmL^2(S)\otimes \alter^k(\bbR^n)$. We denote by $\rmH^r_\rmd\difo^k(S)$ the following space:
\begin{equation}
\rmH^r_\rmd\difo^k(S) = \{u \in \rmH^r\difo^k(S) \ : \ \rmd u \in \rmH^r\difo^{k+1}(S) \}.
\end{equation}
We are interested in the complexes:
\begin{equation}
\xymatrix{
\ldots \ar[r] & \rmH^r_\rmd\difo^{k-1}(S)\ar[r] & \rmH^r_\rmd\difo^k(S)\ar[r] & \rmH^r_\rmd\difo^{k+1}(S)\ar[r] & \ldots
}
\end{equation}
We are also interested in letting $r$ decrease in the complex, at some index, as follows:
\begin{equation}
\xymatrix{
\ldots \ar[r] & \rmH^r_\rmd\difo^{k-1}(S)\ar[r] & \rmH^r\difo^k(S)\ar[r] & \rmH^{r-1}_\rmd\difo^{k+1}(S)\ar[r] & \ldots
}
\end{equation}

If we restrict attention to dimension $n= 2$ and $r =0 , 1$ this leaves us with three possibilities:
\begin{equation}
\xymatrix{
\rmH^1\difo^{0}(S)\ar[r] & \rmH^0_\rmd\difo^1(S)\ar[r] & \rmH^0\difo^{2}(S)
}
\end{equation}

\begin{equation}
\xymatrix{
\rmH^2\difo^{0}(S)\ar[r] & \rmH^1\difo^1(S)\ar[r] & \rmH^0\difo^{2}(S)
}
\end{equation}

\begin{equation}
\xymatrix{
\rmH^2\difo^{0}(S)\ar[r] & \rmH^1_\rmd\difo^1(S)\ar[r] & \rmH^1\difo^{2}(S)
}
\end{equation}
We refer to these sequences as de Rham sequences with regulartity $(1,0+,0)$, $(2,1,0)$ and $(2,1+,1)$ respectively. The two last spaces in the two last sequences are of interest for conforming discretizations of the Stokes equation. It should be pointed out that some reformulations of the Stokes equation with auxilliary variables, can be handled with the first type of sequence (e.g. \cite{Ned82}). There are also examples of non-conforming methods that have been successfull, such as the Crouzeix-Raviart element \cite{Bre15}. As we see it, these methods have been developed because $\rmH^1$-conforming methods, such as those we introduce here, were not known.

We are interested in constructing finite element spaces which provide subcomplexes of the above three complexes. These subcomplexes should be equipped with commuting interpolation operators. For this purpose a framework of FES has been developed for the first type of complex, starting in \cite{Chr08M3AS}.  It is summarized in \cite{ChrRap16}. In this paper, we extend the framework so that it can encompass the other two types of complexes, and more generally, we believe, arbitrary $r \geq 0$ as well as switches between different $r$ as sketched above. For small $r$ we provide examples that illustrate that high order polynomials can be included in the finite element spaces, to achieve arbitrarily high approximation order. In arbitrary dimension we also illustrate that it can be useful to consider different simplicial refinements at different indices of the differential complex. A key tool in our construction is the use of the Poincar\'e operators, as has already been used to construct complexes of regularity $(1,0+,0)$, and generalizations to arbitrary dimension, \cite{Hip99}\cite{ArnFalWin06}. Many more examples than those provided here, should fit in the proposed framework. 

The paper can be seen as a step towards a general theory of discretization of highly continuous fields (sections of vector bundles), in terms of inverse systems of complexes of jets. From this point of view, the present paper provides examples of $r$-jets of order $r=0$ and $r= 1$. This already seems adequate for many of the PDEs we have in mind, since they are at most second order.

The paper is organized as follows. In \S \ref{sec:reg} we relate the regularity of differential forms to their inter-element continuity, expressed with three different restriction operators. In \S \ref{sec:poinc} we recall methods for proving sequence exactness under the exterior derivative, using the Poincare operator and we sketch how it intervenes in finite element constructions. In \S \ref{sec:lowex} we provide four examples of low order composite finite element sequences in space dimension 2. This motivates the framework of generalized finite element systems and gets the machinery started, with respect to higher order polynomials. In \S \ref{sec:gfes} we provide the appropriate notions on generalized FES, leading up to the notion of harmonic interpolator. In \S \ref{sec:twodhigh} we provide, in dimension 2, examples of composite finite element de Rham sequences with enhanced continuity and arbitrarily high degree of polynomials. In \S \ref{sec:toolhigh} we provide some tools for defining composite finite elements in arbitrary space dimension. In particular we define different simplicial refinements and study some continuous piecewise affine forms on them. In \S \ref{sec:spacehigh} we provide a composite finite element de Rham sequence with enhanced continuity and low order polynomials (at most degree two). We also show how such sequences can be branched into Whitney forms at some index. We conclude with some topics for further research.

\section{Restrictions and regularity of differential forms\label{sec:reg}}

\paragraph{Restriction operators adapted to different regularities.}

Consider a simplicial complex $\calT$ on a domain $S$ in a vector space $\bbV$ of dimension $n$. For differential forms which are piecewise smooth with respect to $\calT$ we have:
\begin{itemize}
\item $u \in \rmH^0_\rmd \difo^k(S)$ iff the pullbacks to faces are singlevalued. If $T\in \calT$ is a simplex, \emph{pullback} means here pullback in the sense of differential forms by the injection $T \to S$. It remembers the action of $u$ only on vectors which are tangent to $T$ (see the paragraph leading to (\ref{eq:pullbackdef})).

In terms of vector proxies $\rmH^1_\rmd \difo^k(S)$ corresponds to $\rmL^2(S)$ vectorfields with $\curl$ in $\rmL^2(S)$, for which the pullback corresponds to taking the tangential component of the vectorfield. On the other hand $\rmH^0_\rmd \difo^{n-1}(S)$ corresponds to $\rmL^2(S)$ vectorfields with $\div$ in $\rmL^2(S)$, for which the pullback to codimension 1 faces corresponds to taking the normal component of the vectorfield.

\item $u \in \rmH^1 \difo^k(S)$ iff the traces on faces are singlevalued. Here \emph{trace} means restriction in the usual sense, remembering the action of $u$ on all tangent vectors in $S$ (not only $T$).

For vector proxies this trace operator corresponds to keeping all the compnents of the vectorfields on the faces.

\item $u \in \rmH^1_\rmd \difo^k(S)$ iff the traces on faces of both $u$ and $\rmd u$ are singlevalued on faces. Here the word trace is used with the same meaning as above.
\end{itemize}
It will be convenient to denote by $\rmC^r\difo^k(S)$ the space of $k$-forms on $S$ of class $\rmC^r$ and by $\rmC^r_\rmd\difo^k(S)$ the space of $u  \in \rmC^r\difo^k(S)$ such that $\rmd u \in \rmC^r\difo^{k+1}(S)$.

We interpret the above conditions ensuring various kinds of regularity, by saying that we have defined three types of \emph{restriction} operators.
Explicitely, according to context, the restriction of a differential form $u \in \rmC^r_\rmd\difo^k(S)$ to a face $T$ of $S$ will be:
\begin{itemize}
\item the pullback of $u$, denoted $\pull_T u$, which is in $\rmC^r_\rmd\difo^k(T)$. 
\item the trace of $u$, denoted $\trace_T u$, which is in $\rmC^r(T) \otimes \alter^k(\bbV)$.
\item the double-trace of $u$, written $( \trace_T u, \trace_T \rmd u)$, which is in $\rmC^r(T)\otimes  \alter^k(\bbV) \oplus \rmC^r(T)\otimes \alter^{k+1}(\bbV)$.
\end{itemize}

The framework of FES, introduced in \cite{Chr08M3AS} and developed further in \cite{Chr09AWM}\cite{ChrMunOwr11}\cite{ChrRap16}\cite{ChrGil16} was designed to handle restrictions of the first type, whereas now we are interested in the other cases as well. More generally, we will consider a cellular complex $\calT$ and restrictions from $T$ to $T'$ where $T, T'$ are cells in $\calT$ and $T' \subseteq T$.

\paragraph{Admissibility condition.}
When we start with a $k$-form $u \in \rmC^r_\rmd\difo^k(S)$, the trace of $(u, \rmd u)$ on a cell $T$, also called the double-trace of $u$, is in $\rmC^r(T)\otimes \alter^k(\bbV) \oplus \rmC^r(T)\otimes \alter^{k+1}(\bbV)$, but all elements of the latter sum cannot occur. In other words there are admissibility conditions. In this paragraph we determine them. 

First we introduce some notations:

-- When $v \in \rmC^r(T)\otimes \alter^k(\bbV)$ we denote by $\pull_T v \in \rmC^r\difo^k(T)$ the induced $k$-form on $T$, that remembers the action of $u$ only on vectors in $\bbV$ that are tangential to $T$.

-- When $u$ is a $k$-form on $S$ and $X$ is a vectorfield on $S$, we denote by $u \ctr X$ the contraction of $u$ by $X$, which is the $(k-1)$-form defined at $x \in S$ by:
\begin{equation}
(u \ctr X)_x(\xi_2, \ldots, \xi_{k}) = u_x( X(x), \xi_2, \ldots, \xi_k).
\end{equation}

\begin{lemma}
Let $\bbV$ be a finite dimensional vector space. Let $(e_i)$ be a basis of $\bbV$ and let $(f_i)$ be the dual basis. Then for $u \in \alter^k(\bbV)$, $k \geq 1$, we have:
\begin{equation}
\sum_i f_i \wedge (u \ctr e_i) = k\, u.
\end{equation}
\end{lemma}
\begin{proof}
By induction on $k$.
\end{proof}
We may consider that this identity is true also for $k=0$, the left hand side being $0$ by definition of contraction of $0$-forms.

\begin{proposition} Fix $r \geq 0$.
Let $\bbV$ be a vector space and let $T$ be a subspace. Let $v_0\in \rmC^{r+1}(T)\otimes  \alter^k(\bbV)$ and $v_1 \in \rmC^r(T)\otimes \alter^{k+1}(\bbV)$. The following are equivalent:
\begin{itemize}
\item There exists $u \in \rmC^{r+1} \difo^k(\bbV)$ such that $\trace_T u = v_0$ and $\trace_T \rmd u = v_1$.
\item The induced forms $\pull_T v_0 \in \rmC^{r+1}\difo^k(T)$ and $\pull_T v_1 \in \rmC^r \difo^{k+1}(T)$ (obtained by remembering only the action on tangent vectors to $T$),  are related by:
\begin{equation} 
\rmd \pull_T v_0 = \pull_T v_1.
\end{equation}
\end{itemize} 
\end{proposition}
\begin{proof}
\tpoint The first condition implies the second, because the exterior derivative commutes with pullback. 

\tpoint We prove that the second condition implies the first.  We write $\bbV = T \oplus U$. We introduce a vector field $X$ on $\bbV$, defined by, for any $x \in T$ and any $y \in U$:
\begin{equation}
X(x + y) = y.
\end{equation}
We choose a basis $(e_i)_{i \in I}$ of $T$ and $(e_j)_{j \in J}$ of $U$. We impose $I \cap J = \emptyset$, so they combine to a basis of $\bbV$ and we let $(f_i)_{i \in I \cup J}$ denote the corresponding dual basis of $\bbV$. We let $\partial_i$ denote the directional derivative with respect to $e_i$.

\tpoint We first extend $v_0$ to an element $u$ of $\rmC^{r+1} \difo^k(\bbV)$ by putting $u(x +y) = v_0 (x)$ for $x \in T$ and $y \in U$. Substracting this extension we are left with the the case $v_0 =0$ and $\pull_T v_1=0$. To avoid clutter we denote $v=v_1$.

\tpoint Suppose $v$ is of the form: $v = w\, w_T \wedge w_U$ with $w_U = f_{j_1} \wedge \ldots \wedge f_{j_l}$ (with $l\geq 1$ distinct indices in $J$), $w_T = f_{i_1} \wedge \ldots \wedge f_{i_{k+1-l}}$ (with $k+1-l$ distinct indices in $I$) and $w$ a scalar function on $T$.

We trivially extend $w$ to $\bbV$, which yields an extension of $v$ to a $(k+1)$-form on $\bbV$, which we still denote by $v$. We put $u = v \ctr X$. We write:
\begin{align}
\rmd (v \ctr X) & = \sum_{i} f_i \wedge \partial_i (v \ctr X), \\
& = \sum_{i \in I} f_i \wedge ((\partial_i v) \ctr X) + \sum_{j \in J} f_j \wedge (v \ctr e_j).
\end{align}
The first term here, when restricted to $T$, is zero. For the second term we have:
\begin{align}
\sum_{j \in J} f_j \wedge (v \ctr e_j) & = \sum_{j \in J} f_j \wedge ((-1)^{k+1-l} w w_T \wedge (w_U \ctr e_j))\\
& = w w_T \wedge \sum_{j \in J} f_j \wedge (w_U \ctr e_j),\\
& = l\, v.
\end{align}
We also remark that $v \ctr X$ is zero on $T$. Dividing $v\ctr X$ by $l$, we have a suitable extension of $(0, v)$.

\tpoint Now, in general, the condition $\pull_T v_1= 0$ guarantees that $v_1$ is a linear combination of forms $w\, w_T \wedge w_U$ of the above type, all for some $l \geq 1$. 
\end{proof}

This result motivates the following definition.
\begin{definition}
Let $v_0 \in \rmC^0(T) \otimes \alter^k(\bbV)$ and $v_1 \in \rmC^0(T) \otimes \alter^{k+1}(\bbV)$. We say that the pair $(v_0,v_1)$ is \emph{admissible} if $\rmd \pull_T v_0 = \pull_T v_1$, where $\rmd \pull_T v_0$ is defined a priori in the sense of distributions.
\end{definition}

\paragraph{On the necessity of composite elements.}

Consider the line $T= \bbR \times \{0\}$ sitting in $\bbV = \bbR^2$. The preceding paragraph shows that in order to extend data on $T$, consisting of a pair $(v_0, v_1) \in C^{r+1}(T) \oplus C^r(T) \otimes \alter^1(\bbV)$,  to a function in $\rmC^{r+1}(\bbV)$, there is the compatibility condition $\rmd v_0 = \pull_T v_1$. We now illustrate that if several lines meet at a vertex (which will be the case in simplicial complexes), additional compatibility conditions could appear at the vertex, if we require the extension to be at least $\rmC^2(\bbV)$.

Suppose we have two coordinates $(x,y)$. We have data consisting of functions $p_0$, $p_1$ on the $x$-axis which are $\rmC^1(\bbR)$ and $\rmC^0(\bbR)$ respectively and as well as functions $q_0$, $q_1$ on the $y$-axis that are $\rmC^1(\bbR)$ and $\rmC^0(\bbR)$ respectively.

We want to find a function $u$ on $\bbR^2$ of class $\rmC^1(\bbR^2)$ such that $(u, \partial_y u)$ restricts to $(p_0, p_1)$ on the $x$-axis and $(u, \partial_x u)$ restricts to $(q_0, q_1)$ on the $y$-axis. There are compatibily conditions at the origin:
\begin{equation}
(p_0(0), \dot p_0(0), p_1(0)) = (q_0(0), q_1(0), \dot q_0(0)).
\end{equation}
These are sufficient for the existence of a $\rmC^1(\bbR^2)$ extension.

However, for extensions of class $\rmC^2(\bbR^2)$ of the same data, there is an additional constraint, expressing that $\partial_x \partial_y u = \partial_y \partial_x u$ at the origin, namely:
\begin{equation}
\dot p_1(0) = \dot q_1(0).
\end{equation}
This remark applies in particular to polynomials. Compare with the fact that the Argyris element is $\rmC^2$ at vertices, even though one only wants to obtain $\rmC^1$ functions.

In this paper we are not interested in constructing functions that are globally $\rmC^2$. We want $\rmC^1$ functions, glued together from data on subsimplices that only involve derivatives up to order $1$.

This explains why we prefer to construct spaces in terms of composite polynomials: we can then hope to satisfy first order constraints (that guarantee $\rmC^1$ continuity), without adding second order constraints (corresponding for instance to symmetry of mixed derivatives as above). Another choice could have been to use rational functions that are $\rmC^1$ on the simplices but not $\rmC^2$.

\paragraph{A differential acting on admissible pairs.}

Let $T$ be a flat cell in a vectorspace $\bbV$. Suppose that we have subspaces $B^k(T)$ of $\rmC^0(T) \otimes \alter^{k}(\bbV)$, such that the exterior derivative on $T$ maps $ \pull_T B^k(T)$ into $\pull_T B^{k+1}(T)$. Then we define the following spaces of admissible pairs:
\begin{equation}
A^k(T) = \{(v_0, v_1) \in B^k(T) \oplus B^{k+1}(T) \ : \ \rmd \pull_T v_0 = \pull_T v_1 \}.
\end{equation}

We define the following differential:
\begin{equation}\label{eq:difadm}
\sfd^k: \mapping{A^k(T)}{A^{k+1}(T)}{(v_0,v_1)}{(v_1, 0)}
\end{equation}
It is well defined, because if $(v_0,v_1)$ is admissible then $\rmd \pull_T(v_1) =  \rmd^2 \pull_T v_0= 0$, so $(v_1, 0)$ is admissible. Moreover we see that $\sfd^{k+1} \circ \sfd^k =0$.

\begin{lemma}\label{lem:exind}
The sequence:
\begin{equation}
A^k(T) \to A^{k+1}(T) \to A^{k+2}(T),
\end{equation}
is exact if and only if the sequence:
\begin{equation}
\pull_T B^{k}(T) \to \pull_T B^{k+1}(T) \to \pull_T B^{k+2}(T), 
 \end{equation}
is exact.
\end{lemma}

\begin{proof}
\tpoint Suppose the second sequence is exact.

Given an admissible $(v_1,0)\in A^{k+1}(T)$ we have $\rmd \pull_T v_1 = 0$. Choose $v_0' \in \pull_T B^k(T)$ such that $\rmd v_0' = v_1$ and then $v_0 \in B^k(T)$ such that $\pull_T v_0 = v_0'$. Then $(v_0,v_1)$ is admissible and maps to $(v_1,0)$. 

\tpoint Suppose the first sequence is exact.

Suppose $v_1' \in \pull_T B^{k+1}(T)$ satisfies $\rmd v_1'=0$. Choose $v_1 \in B^{k+1}(T)$ such that $\pull_T v_1 = v_1'$. Then $(v_1,0) \in B^{k+1}(T)$ and $\sfd (v_1,0) = 0$. Writing $(v_1,0) = \sfd (v_0,v_1)$ we get $v_0 \in B^{k}(T)$ such that $(v_0, v_1)$ is admissible. Then $v_0' = \pull_T v_0 \in \pull_T B^{k}(T) $ satisfies $\rmd v_0' = v_1'$.
\end{proof}

\section{Poincar\'e and Koszul operators.\label{sec:poinc}}

\paragraph{Poincar\'e operators.}
We recall some properties of the so-called Poincar\'e and Koszul operators, used for constructing finite element differential forms in  \cite{Hip99} and \cite{ArnFalWin06} respectively. For the former, we refer to \cite{Lan99}, especially chapter V, but recall the main steps of interest to us.

Recall that when $S$ and $S'$ are domains and $\Phi : S \to S'$ is differentiable, the pullback of a $k$-form $u$ on $S'$, by $\Phi$, is the $k$-form $\Phi^\star u$ on $S$ defined at $x \in S$ by:
\begin{equation}\label{eq:pullbackdef}
(\Phi^\star u)_x (\xi_1, \ldots, \xi_k) = u_{\Phi(x)}(\rmD \Phi(x) \xi_1, \ldots, \rmD \Phi(x) \xi_k).
\end{equation}

Suppose now that $S$ is a domain. We consider a smooth map $F:[0,1] \times S \to S$, and interpret it as a family of maps $F_t = F(t,\cdot): S \to S$, for $t \in [0,1]$, defining a homotopy between $F_0$ and $F_1$. We write:
\begin{equation}\label{eq:intder}
F_1^\star u - F_0^\star u = \int_0^1 \partial_t (F_t^\star u) \rmd t.
\end{equation}
For most $t \in [0,1]$, we suppose we have a vector field $G_t$ on $S$ such that, for $x \in S$:
\begin{equation}
G_t(F_t(x)) = \partial_t F_t(x).
\end{equation}
This uniquely defines $G_t$ on $S$ when $F_t : S \to S$ is a diffeomorphism, and expresses that any curve $F_\bs(x)$ flows with $G_\bs$.

When $u$ is a $k$-form we have:
\begin{align}
\partial_t (F_t^\star u) &= F_t^\star \lie_{G_t} u,\\
& = F_t^\star ((\rmd u) \ctr G_t + \rmd (u \ctr G_t)),
\end{align}
using Cartan's formula for the Lie derivative.

The Poincar\'e operator associated with $F$ (and $G$), acting on differential $k$-forms, is denoted $\poincare[F]$ or, when the choice of $F$ is clear, as $\poincare$. It can be written succintly:
\begin{equation}
\poincare[F] u= \int_0^1 F_t^\star (u \ctr G_t) \rmd t.
\end{equation}
More explicitely, if $u$ is a $k$-form:
\begin{equation}
(\poincare[F] u)_x(\xi_2, \ldots, \xi_k) = \int_0^1 u_{F_t(x)}(\partial_tF_t(x), \rmD F_t(x)\xi_2, \ldots, \rmD F_t(x) \xi_k) \rmd t.
\end{equation}
With these considerations in mind, (\ref{eq:intder}) can be expressed with the Poincar\'e operator as follows:
\begin{equation}
F_1^\star u - F_0^\star u = \poincare[F]  \eder u + \eder  \poincare[F] u. 
\end{equation}
Suppose that $F_1$ is the identity on $S$ and that $F_0$ is constant. Then the formula gives, for $\poincare = \poincare[F]$ acting on $k$-forms with $k\geq 1$:
\begin{equation}\label{eq:poincare1}
\id = \poincare \eder + \eder \poincare,
\end{equation}
whereas if $u$ is a function, considered as a $0$-form, and the value of $F_0$ is $W$, we get:
\begin{equation}\label{eq:poincare2}
u - u(W) = \poincare \eder u.
\end{equation}
If now $S$ is a domain in an affine space, which is starshaped with respect to, say $W$, we may choose $F$ to be defined by:
\begin{equation}
F_t(x) = tx + (1-t)W.
\end{equation}
Then we may substitute in the above formulas:
\begin{equation}
\partial_t F_t(x) = x - W, \ G_t(y) = \frac{1}{t} (y - W), \textrm{ and } \rmD F_t(x) (\xi) = t \xi.
\end{equation}
We denote the associated Poincar\'e operator as $\poincare_W$. It is defined explicitely on $k$-forms $u$ by:
\begin{equation}\label{eq:poinexpl}
(\poincare_W u)_x(\xi_2 \ldots, \xi_k) = \int_0^1 t^{k-1} u_{W + t(x-W)} (x - W, \xi_2, \ldots, \xi_k) \rmd t.
\end{equation}

\paragraph{Koszul operators.}

In an affine space, given a choice of a point $W$, we may also define directly a vector field $X_W$ by:
\begin{equation}
X_W \ : \  x \mapsto x- W.
\end{equation}
The contraction of a differential form by $X_W$ is called the Koszul operator associated with $W$ and denoted:
\begin{equation}
\koszul_W \ : \  u \mapsto \koszul_W  u = u \ctr X_W.
\end{equation}
If the choice of $W$ is clear from the context, we may sometimes omit it from the notation.

If $u$ is a $k$-form which, with respect to some choice of origin $W$ and basis, has components which are homogeneous polynomials of degree $r$, then from (\ref{eq:poinexpl}) we get:
\begin{equation}
\poincare_W u = \frac{1}{ k + r} \koszul_W u,
\end{equation}
which is polynomial and whose components are homogeneous of degree $r+1$.

We are mainly interested in identitites (\ref{eq:poincare1},\ref{eq:poincare2}) and knowing that the Poincar\'e operator maps polynomials to polynomials, increasing degree by only one. Sometimes explicit computations are more handy with the Koszul operator. For composite elements it will be important where we locate $W$, so as to respect the refinement used.

\begin{remark}
From the above discussion of Poincar\'e operators, we can derive the identity, on $k$-forms which are homogeneous polynomials of degree $r$:
\begin{align}
(\eder \koszul_W + \koszul_W \eder) u & = (r+k) \eder \poincare_W u + (r-1 + k +1) \poincare_W \eder u,\\
& = (r+k) u.
\end{align}
It is obtained by two somewhat different techniques in section 3.2 of \cite{ArnFalWin06}.
\end{remark}

\paragraph{Complexes constructed with Poincar\'e and Koszul operators.}
We suppose we have a  complex $U^\bs$ (of perhaps infinite dimensional spaces):

\begin{equation}
\xymatrix{
\ldots \ar[r]^{\rmd_{k-2}} & U^{k-1} \ar[r]^{\rmd_{k-1}} & U^k \ar[r]^{\rmd_k} & U^{k+1} \ar[r]^{\rmd_{k+1}} & \ldots
}
\end{equation}
 
We also suppose that we have operators $\poincare_k : U^k \to U^{k-1}$ such that:
\begin{equation}\label{eq:homz}
\poincare_{k+1} \eder_k + \eder_{k-1} \poincare_{k} = \lambda_k \id_{U^k},
\end{equation}
where $\lambda_k$ is a non-zero scalar. It follows that the complex $U^\bs$ is exact.

We suppose furthermore that:
\begin{equation}\label{eq:pzz}
\poincare_{k-1} \poincare_{k} = 0.
\end{equation}

In this situation we suppose that we have subspaces $V^\bs$ that form a complex:
\begin{equation}
\xymatrix{
\ldots \ar[r]^{\rmd_{k-2}} & V^{k-1} \ar[r]^{\rmd_{k-1}} & V^k \ar[r]^{\rmd_k} & V^{k+1} \ar[r]^{\rmd_{k+1}} & \ldots
}
\end{equation}

We then define:
\begin{equation}
W^k = V^k + \poincare_{k+1} V^{k+1}.
\end{equation}

\begin{proposition}
The spaces $W^\bs$ form an exact complex. We have:
\begin{equation}\label{eq:dpw}
W^k = \rmd_{k-1} W^{k-1} \oplus \poincare_{k+1} W^{k+1}.
\end{equation}

\end{proposition}
\begin{proof}
\tpoint We notice that for $u \in V^{k+1}$ we have:
\begin{align}
\eder_k \poincare_{k+1} u & = \lambda_k u - \poincare_{k+2} \eder_{k+1} u,\\
& \in V^{k+1} + \poincare_{k+2}V^{k+2} = W^{k+1}.
\end{align}
Therefore $\rmd_{k}$ maps $W^k$ to $W^{k+1}$.

\tpoint We also see that $\poincare_k$ maps $W^k$ to $W^{k-1}$, using (\ref{eq:pzz}). Therefore identity (\ref{eq:homz}) also holds for the complex $W^\bs$. It follows that it is exact and that we have:
\begin{equation}
W^k = \rmd_{k-1} W^{k-1} + \poincare_{k+1} W^{k+1}.
\end{equation}
Finally, if $u \in \rmd_{k-1} W^{k-1}  \cap \poincare_{k+1} W^{k+1}$, then $\rmd_k u = 0$ and $\poincare_k u = 0$ so that $u = 0$, also from (\ref{eq:homz}).

\end{proof}

\begin{remark}
The spaces $W^\bs$ form a cochain complex with respect to $\eder_\bs$, but they also form a chain complex with respect to $\poincare_\bs$, and it is exact.
\end{remark}

\paragraph{Examples}

We can take $V^k = \poly^p \difo^k(\bbR^n)$.
Then we get the exact complex of spaces:
\begin{equation}
W^k = \poly^{p+1}_-\difo^k(\bbR^n) =  \poly^p \difo^k(\bbR^n) + \poincare \poly^p \difo^k(\bbR^n).
\end{equation}
This generalizes the first family of N\'ed\'elec-Raviart-Thomas, and $p= 0$ corresponds to Whitney forms.

We can also take $V^k = \poly^{p-k} \difo^k(\bbR^n)$. Since it is stable under $\poincare$ it is exact. This generalizes the second family of N\'ed\'elec-Brezzi-Douglas-Marini.

We now consider the construction of composite elements on a simplicial complex. For instance, on a triangulation, the Cloch-Tocher split consists in adding one point to each triangle, and join it with the three vertices, so that each triangle is divided into three smaller triangles.

More generally one can consider a simplicial complex where each simplex is included in an $n$-dimensional simplex. We suppose that we add an inpoint to each $n$-dimensional simplex and join it to the vertices, and possibly inpoints of boundary simplices. More precisely we suppose here that a simplicial refinement of the $(n-1)$-skeleton is chosen. For each $n$-dimensional simplex $T$ the inpoint $W$ is coned with the refinement of the boundary of $T$. In this paragraph we denote such a refinement by $\calS$.

It is then natural to define finite element spaces on $T$ consisting of piecewise polynomials with respect to $\calS$, using the Poincar\'e operator associated with $W$.

We can then take $V^k(T) = \rmC^0_\rmd \poly^{p-k} \difo^k(\calS)$, consisting of $k$-forms which are piecewise polynomials of degree $p-k$, that are continuous and with continuous exterior derivative. The Poincar\'e operator associated with the inpoint maps
 $V^k(T)$ to $V^{k-1}(T)$ so that we get an exact complex. This construction resembles that of to the second family above. We carry out this construction in dimension $n= 2$ in Section \ref{sec:twodhigh}.

Another construction allows to have different simplicial refinements of $T$ for each index $k$, and resembles that of the first family above. Let's call the refinements of $T$, $\calS_k$.
We can define:
\begin{equation}
K^k(T) = \{ u \in \rmC^0 \poly^{p} \difo^k(\calS_k) : \rmd u = 0 \}.
\end{equation}
These spaces form a complex which is not exact. We can then define the augmented spaces:
\begin{equation}
A^k(T) = K^k(T) + \poincare_W K^{k+1}(T).
\end{equation}
Notice that $A^k(T)$ contains $\poly^p \difo^k(T)$. We carry out a construction of this type in arbitrary dimension $n$, with $p = 1$, in Section \ref{sec:spacehigh}.

We also show how one can branch such spaces into standard Whitney forms, by augmenting the complex:
\begin{equation}
\xymatrix{
\ldots \ar[r]^{\rmd_{k-2}} & K^{k-1} \ar[r]^{\rmd_{k-1}} & K^k \ar[r]^{\rmd_k} & \difo^{k+1} \ar[r]^{\rmd_{k+1}} & \ldots
}
\end{equation}
See in (\ref{eq:adefkbis}) how this leads to a new space at index $k$.

%%%%%%%%%%%%%%%%%%%%%%%%%%%%%%%%%%%%%%%
\section{Low order finite element complexes in 2D \label{sec:lowex}}

We proceed to define four complexes based on the Clough-Tocher element.

\paragraph{A complex of regularity $(2,1+,1)$.}

Let $T$ be a triangle with vertices $V_0, V_1$ and $V_2$. Choose a point $W$ in the interior of $T$, and subdivide $T$ into three triangles, by drawing edges from $W$ to $V_0$, $V_1$ and $V_2$. This equips $T$ with a simplicial refinement, which we denote by $\calR$.

The Clough-Tocher element involves a degree of freedom on the edges, which can be taken as the normal derivative at the midpoint. More generally, for each edge $E$, we consider a linear form on one-forms, which evaluates the one-form in a transverse direction $\nu_E$,  at an interior point $W_E$ of the edge. We denote this degree of freedom on $1$-forms as $\mu_E$.

For the following proposition we also refer to Figure \ref{fig:lowone}.
\begin{proposition}\label{prop:lowone}
We have an exact sequence:
\begin{equation}
\xymatrix{
0 \ar[r] & \bbR \ar[r] & \rmC^1 \poly^3 \difo^0(\calR) \ar[r] & \rmC^0_\rmd \poly^2 \difo^1(\calR) \ar[r]  & \rmC^0\poly^1 \difo^2(\calR)  \ar[r] & 0
}
\end{equation}
The spaces are, more explicitely, the following:
\begin{itemize}
\item $\rmC^1 \poly^3 \difo^0(\calR)$ consists of piecewise $\poly^3$ functions, which are of class $\rmC^1(T)$.
\item $\rmC^0_\rmd \poly^2 \difo^1(\calR)$ consists of piecewise $\poly^2$ one-forms which are $\rmC^0(T)$ with exterior derivative in $\rmC^0(T)$.
\item $\rmC^0\poly^1 \difo^2(\calR)$ consists of piecewise $\poly^1$ two-forms, which are $\rmC^0(T)$.
\end{itemize}
Moreover these spaces have the following properties:
\begin{itemize}
\item $\rmC^1 \poly^3 \difo^0(\calR)$ has dimension $12$. Any element $u$ is determined by the following data:
\begin{itemize} 
\item vertices $V$: one DoF for $u(V)$ and two DoFs for $\rmd u(V)$.
\item edges $E$:  one DoF, say $\mu_E (\rmd u)$.
\end{itemize}
\item $\rmC^0_\rmd \poly^2 \difo^1(\calR)$ has dimension $15$. Any element $u$ is determined by:
\begin{itemize}
\item vertices $V$: two DoFs for $u(V)$ and one DoF for $\rmd u (V)$.
\item edges $E$: two DoFs, tranverse and tangential:  $\mu_E(u)$ and $\int_E u$.
\end{itemize}
\item $\rmC^0\poly^1 \difo^2(\calR)$ has dimension $4$. Any element $u$ is determined by:
\begin{itemize}
\item vertices $V$: one DoF for $u(V)$.
\item interior $T$: one DoF, namely the integral $\int_T u$.
\end{itemize}
\end{itemize}
The above degrees of freedom provide commuting interpolators.
\end{proposition}

\begin{proof}
\tpoint Exactness of the complex can be deduced from the Poincar\'e operator associated with the inpoint $W$. It maps the spaces one to the other.\\
Notice by the way that we get the identity:
\begin{align}\label{eq:czdopz}
 \rmC^0_\rmd \poly^2 \difo^1(\calR) & = \rmd \rmC^{1} \poly^{3} \difo^0 (\calR )\oplus \poincare_W \rmC^{0}\poly^{1}\difo^2(\calR),\\
 & \approx \curl \rmC^{1} \poly^{3} (\calR )\oplus X_W \rmC^{0}\poly^{1}(\calR).
\end{align}

\tpoint Counting constraints on the space of piecewise polynomials of degree 3 on $\calR$, shows that the dimension of the first space is at least $30 - 18 = 12$. That the dimension is exactly $12$ follows from proving unisolvence of the DoFs, which is done in particular in \cite{Cia74}\cite{Per76}. 

It amounts to showing that if both $u$ and its derivatives are $0$ on $\partial T$, then $u= 0$. Such a $u$ can be written $\lambda_W^2 v$ where $v \in \rmC^0\poly^1(\calR)$, where $\lambda_W$ is the barycentric coordinate map of $\calR$ associated with the inpoint $W$. We have that:
\begin{equation}
\rmd u = 2 \lambda_W v \rmd \lambda_W + \lambda_W^2 \rmd v.
\end{equation}
Since $u\in \rmC^1(T)$ we get that the following form is continuous on $T$:
\begin{equation}
2 v \rmd \lambda_W + \lambda_W \rmd v.
\end{equation}
Since $\rmd \lambda_W$ is discontinuous at the vertices, the three vertex values of $v$ are $0$, so that $v$ is proportional to $\lambda_W$. Since $\rmd \lambda_W$ is discontinuous at $W$ we deduce $u = 0$.

\tpoint The last space has dimension $4$ and the given degrees of freedom are unisolvent.

\tpoint Counting constraints on the spaces of piecewise polynomial one-forms, shows that the dimension of the second space is at least $36 - 21 = 15$. If $u \in \rmC^0_\rmd \poly^2 \difo^1(\calR)$ has degrees of freedom $0$ we write:
\begin{equation}
u = \poincare_W \eder u + \eder \poincare_W u.
\end{equation}
We notice that $\rmd u \in \rmC^0\poly^1 \difo^2(\calR)$ and has degrees of freedom $0$ so $\rmd u = 0$.
We also notice that $v= \poincare_W u$ satisfies $\rmd v = u$. Its degrees of freedom are $0$ except perhaps the vertex values $v(V)$. They must be the same, because $\int_E \rmd v  = 0$ for each edge $E$. Hence $v$ is constant, so $u = \rmd v =0$.

This proves unisolvence and that the dimension count is exact (the dimension can also be deduced from (\ref{eq:czdopz})).

\tpoint It is straightforward to check that the interpolator associated with these DoFs commutes with the exterior derivative. 
\end{proof}

What remains in order to prove that this is a good finite element, is that \emph{inter}-element continuity behaves as expected. On edges the spaces of restrictions from adjacent triangles should be the same. 

\begin{center}

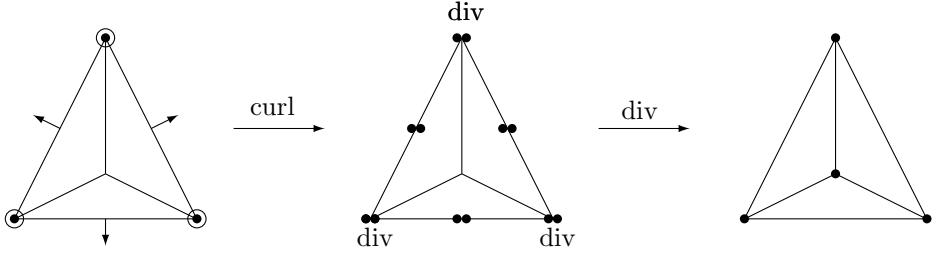
\begin{figure}

\setlength{\unitlength}{1.2cm}
\begin{picture}(2,2)(-0.8,0)
\put(0,0){
\begin{picture}(2,2)
\put(-1, 0){\line(1,2){1}} 
\put(0, 2){\line(1,-2){1}}
\put(-1,0){\line(1,0){2}}
\put(-1.,0){\circle*{0.1}}
\put(-1.,0){\circle{0.2}}
\put(1.,0){\circle*{0.1}}
\put(1.,0){\circle{0.2}}
\put(0,2){\circle*{0.1}}
\put(0,2){\circle{0.2}}
\put(0, 0.5){\line(0, 1){1.5}}
\put(0, 0.5){\line(2, -1){1}}
\put(0, 0.5){\line(-2, -1){1}}
\put(0.5, 1){\vector(2,1){0.3}}
\put(-0.5, 1){\vector(-2,1){0.3}}
\put(0, 0){\vector(0,-1){0.3}}
\end{picture}
}

\put(1.5, 1){\vector(1, 0){1}}
\put(1.68, 1.15){$\curl$}

\put(4,0){
\put(-1, 0){\line(1,2){1}} 
\put(0, 2){\line(1,-2){1}}
\put(-1,0){\line(1,0){2}}
\put(-0.05,2){\circle*{0.1}}
\put(0.05,2){\circle*{0.1}}
\put(-0.15,2.2){$\div$}
\put(-0.15,2.2){$\div$}
\put(-1.15,-0.3){$\div$}
\put(0.85,-0.3){$\div$}

\put(-1.05,0){\circle*{0.1}}
\put(-0.95,0){\circle*{0.1}}
\put(-1.05,0){\circle*{0.1}}
\put(-0.95,0){\circle*{0.1}}
\put(1.05,0){\circle*{0.1}}
\put(0.95,0){\circle*{0.1}}
\put(0.55, 1){\circle*{0.1}}
\put(0.45, 1){\circle*{0.1}}
\put(-0.55, 1){\circle*{0.1}}
\put(-0.45, 1){\circle*{0.1}}

\put(-0.05, 0){\circle*{0.1}}
\put(0.05, 0){\circle*{0.1}}
\put(0, 0.5){\line(0, 1){1.5}}
\put(0, 0.5){\line(2, -1){1}}
\put(0, 0.5){\line(-2, -1){1}}
}

\put(5.5, 1){\vector(1, 0){1}}
\put(5.75, 1.1){{$\div$}}

\put(8,0){
\begin{picture}(2,2)
\put(-1, 0){\line(1,2){1}} 
\put(0, 2){\line(1,-2){1}}
\put(-1,0){\line(1,0){2}}
\put(0, 0.5){\line(0, 1){1.5}}
\put(0, 0.5){\line(2, -1){1}}
\put(0, 0.5){\line(-2, -1){1}}

\put(-1,0){\circle*{0.1}}
\put(1,0){\circle*{0.1}}
\put(0,2){\circle*{0.1}}
\put(0,0.5){\circle*{0.1}}

\end{picture}
}

%\put(9, 1){\vector(1, 0){1}}
\end{picture}
\caption{\label{fig:lowone} Clough-Tocher complex with continuous pressure described in Proposition \ref{prop:lowone}.
}
\end{figure}

\end{center}

\begin{remark}
The space $\rmC^0_\rmd \poly^2 \difo^1(\calR)$ is also described in \cite{AlfSor16}, where it is analysed with Bernstein-Bezier techniques. Their definition incorporates the fact that an element of $\rmC^0_\rmd \poly^2 \difo^1(\calR)$ is automatically $\rmC^1$ at the inpoint $W$.
\end{remark}

\paragraph{A minimal complex of regularity $(2,1+,1)$.}

It is also possible, in the previous example, to eliminate the edge degrees of freedom in $\rmC^1 \poly^3 \difo^0(\calR)$, by requiring $\rmd u \ctr \nu_E$ to be affine on edge $E$. Usually one imposes the normal derivative on edges to be affine. This is called the \emph{reduced} HCT element. The transverse edge degree of freedom in $\rmC^0_\rmd \poly^2 \difo^1(\calR)$  is then also eliminated by requiring $u \ctr \nu_E$ to be affine. See Figure \ref{fig:lowonemin}.

The natural degrees of freedom provide a commuting interpolator.

\begin{remark}
We see that we have as many degrees of freedom left for the space of $0$-forms (namely three times the number of vertices) as for the space of $2$-forms (namely the number of vertices plus number of triangles), up to the Euler-Poincar\'e characteristic of the surface. This can be interpreted as a balancing of the degrees of freedom describing the curl and the divergence of the vector fields.
\end{remark}

This complex is minimal, among complexes with this regularity, in a sense which can be made precise in the framework of finite element systems. This is described below, in the last paragraph of \S \ref{sec:gfes}.

\begin{center}
\begin{figure}

\setlength{\unitlength}{1.2cm}
\begin{picture}(2,2)(-0.8,0)
\put(0,0){
\begin{picture}(2,2)
\put(-1, 0){\line(1,2){1}} 
\put(0, 2){\line(1,-2){1}}
\put(-1,0){\line(1,0){2}}
\put(-1.,0){\circle*{0.1}}
\put(-1.,0){\circle{0.2}}
\put(1.,0){\circle*{0.1}}
\put(1.,0){\circle{0.2}}
\put(0,2){\circle*{0.1}}
\put(0,2){\circle{0.2}}
\put(0, 0.5){\line(0, 1){1.5}}
\put(0, 0.5){\line(2, -1){1}}
\put(0, 0.5){\line(-2, -1){1}}
\end{picture}
}

\put(1.5, 1){\vector(1, 0){1}}
\put(1.68, 1.15){$\curl$}

\put(4,0){
\put(-1, 0){\line(1,2){1}} 
\put(0, 2){\line(1,-2){1}}
\put(-1,0){\line(1,0){2}}
\put(-0.05,2){\circle*{0.1}}
\put(0.05,2){\circle*{0.1}}
\put(-0.15,2.2){$\div$}
\put(-0.15,2.2){$\div$}
\put(-1.15,-0.3){$\div$}
\put(0.85,-0.3){$\div$}

\put(-1.05,0){\circle*{0.1}}
\put(-0.95,0){\circle*{0.1}}
\put(-1.05,0){\circle*{0.1}}
\put(-0.95,0){\circle*{0.1}}
\put(1.05,0){\circle*{0.1}}
\put(0.95,0){\circle*{0.1}}
\put(-0.5, 1){\vector(-2, 1){0.3}}
\put(0.5, 1){\vector(2, 1){0.3}}

\put(0, 0){\vector(0, -1){0.3}}
\put(0, 0.5){\line(0, 1){1.5}}
\put(0, 0.5){\line(2, -1){1}}
\put(0, 0.5){\line(-2, -1){1}}
}

\put(5.5, 1){\vector(1, 0){1}}
\put(5.75, 1.1){{$\div$}}

\put(8,0){
\begin{picture}(2,2)
\put(-1, 0){\line(1,2){1}} 
\put(0, 2){\line(1,-2){1}}
\put(-1,0){\line(1,0){2}}
\put(0, 0.5){\line(0, 1){1.5}}
\put(0, 0.5){\line(2, -1){1}}
\put(0, 0.5){\line(-2, -1){1}}

\put(-1,0){\circle*{0.1}}
\put(1,0){\circle*{0.1}}
\put(0,2){\circle*{0.1}}
\put(0,0.5){\circle*{0.1}}

\end{picture}
}

%\put(9, 1){\vector(1, 0){1}}
\end{picture}
\caption{\label{fig:lowonemin} Minimal Clough-Tocher complex with continuous pressure.}
\end{figure}
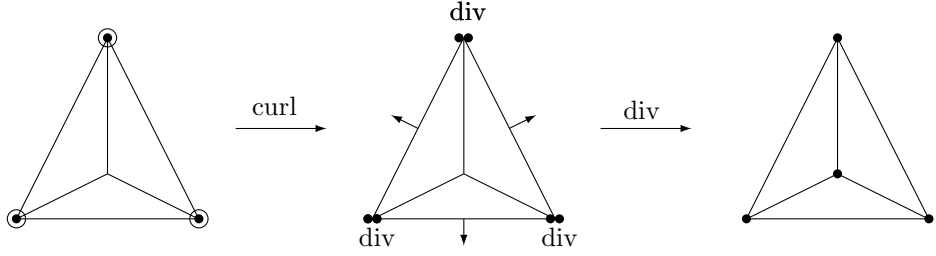
\end{center}

\paragraph{A complex of regularity $(2,1,0)$.}

We may also consider the sequence:
\begin{equation}
\xymatrix{
0 \ar[r] & \bbR \ar[r] & \rmC^1 \poly^3 \difo^0(\calR) \ar[r] & \rmC^0 \poly^2 \difo^1(\calR) \ar[r]  & \poly^1 \difo^2(\calR)  \ar[r] & 0
}
\end{equation}
The spaces are, more explicitely, the following:
\begin{itemize}
\item $\rmC^1 \poly^3 \difo^0(\calR)$ consists of piecewise $\poly^3$ functions, which are of class $\rmC^1(T)$.
\item $\rmC^0 \poly^2 \difo^1(\calR)$ consists of piecewise $\poly^2$ one-forms which are $\rmC^0(T)$.
\item $\poly^1 \difo^2(\calR)$ consists of piecewise $\poly^1$ two-forms.
\end{itemize}
The second space has dimension 20. The last one has dimension 9. Exactness follows from using the Poincar\'e operator at the inpoint $W$. A preliminary reasoning shows that $\rmC^0 \poly^2 \difo^1(\calR)$ should have 2 degrees of freedom per vertex, 2 per edge and 8 interior ones. However it is not clear what they should be, if one wants the interpolator to commute with the exterior derivative. A part from a choice of degrees of freedom adapted to Stokes, these spaces are well known.

\begin{center}
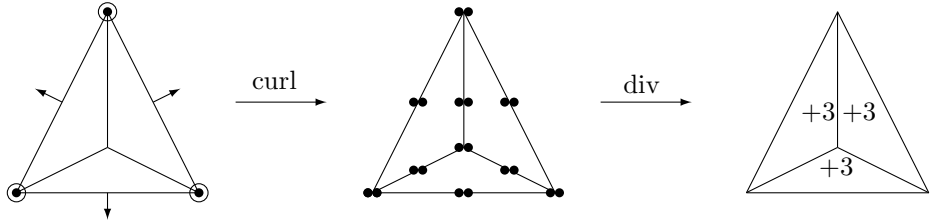
\begin{figure}
\setlength{\unitlength}{1.2cm}
\begin{picture}(2,2)(-0.8,0)
\put(0,0){
\begin{picture}(2,2)
\put(-1, 0){\line(1,2){1}} 
\put(0, 2){\line(1,-2){1}}
\put(-1,0){\line(1,0){2}}
\put(-1.,0){\circle*{0.1}}
\put(-1.,0){\circle{0.2}}
\put(1.,0){\circle*{0.1}}
\put(1.,0){\circle{0.2}}
\put(0,2){\circle*{0.1}}
\put(0,2){\circle{0.2}}
\put(0, 0.5){\line(0, 1){1.5}}
\put(0, 0.5){\line(2, -1){1}}
\put(0, 0.5){\line(-2, -1){1}}
\put(0.5, 1){\vector(2,1){0.3}}
\put(-0.5, 1){\vector(-2,1){0.3}}
\put(0, 0){\vector(0,-1){0.3}}
\end{picture}
}

\put(1.5, 1){\vector(1, 0){1}}
\put(1.68, 1.15){$\curl$}

\put(4,0){
\put(-1, 0){\line(1,2){1}} 
\put(0, 2){\line(1,-2){1}}
\put(-1,0){\line(1,0){2}}
\put(-0.05,2){\circle*{0.1}}
\put(0.05,2){\circle*{0.1}}
\put(-1.05,0){\circle*{0.1}}
\put(-0.95,0){\circle*{0.1}}
\put(-1.05,0){\circle*{0.1}}
\put(-0.95,0){\circle*{0.1}}
\put(1.05,0){\circle*{0.1}}
\put(0.95,0){\circle*{0.1}}
\put(0.55, 1){\circle*{0.1}}
\put(0.45, 1){\circle*{0.1}}
\put(-0.55, 1){\circle*{0.1}}
\put(-0.45, 1){\circle*{0.1}}

\put(-0.05, 1){\circle*{0.1}}
\put(0.05, 1){\circle*{0.1}}
\put(-0.05, 0){\circle*{0.1}}
\put(0.05, 0){\circle*{0.1}}
\put(0, 0.5){\line(0, 1){1.5}}
\put(0, 0.5){\line(2, -1){1}}
\put(0, 0.5){\line(-2, -1){1}}
\put(-0.45, 0.25){\circle*{0.1}}
\put(-0.55, 0.25){\circle*{0.1}}
\put(0.45, 0.25){\circle*{0.1}}
\put(0.55, 0.25){\circle*{0.1}}
\put(0.05, 0.5){\circle*{0.1}}
\put(-0.05, 0.5){\circle*{0.1}}
}

\put(5.5, 1){\vector(1, 0){1}}
\put(5.75, 1.1){{$\div$}}

\put(8,0){
\begin{picture}(2,2)
\put(-1, 0){\line(1,2){1}} 
\put(0, 2){\line(1,-2){1}}
\put(-1,0){\line(1,0){2}}
\put(0, 0.5){\line(0, 1){1.5}}
\put(0, 0.5){\line(2, -1){1}}
\put(0, 0.5){\line(-2, -1){1}}
\put(-0.2, 0.2){+3}
\put(-0.4, 0.8){+3}
\put(0.05, 0.8){+3}
\end{picture}
}

\end{picture}
\caption{Clough-Tocher complex with discontinuous pressure.\newline
The figure shows the lowest order case: the first space is piecewise cubic, the second is continuous piecewise quadratic and the third is piecewise linear. }
\end{figure}
\end{center}

\paragraph{A minimal complex of regularity $(2,1,0)$.}

In the last example, the spaces are bigger than necessary. A smaller complex of the form:
\begin{equation}
\xymatrix{
0 \ar[r] & \bbR \ar[r] & A^0(T) \ar[r] & A^1(T) \ar[r]  & A^2(T)  \ar[r] & 0
}
\end{equation}
may  be defined as follows. The spaces are:
\begin{itemize}
\item $A^0(T)$ is reduced HCT, of dimension 9.\\
The DoFs are vertex values and vertex values of the exterior derivative.
\item $A^1(T) = \rmd A^0(T) + \poincare_W A^2(T) \approx \curl A^0(T) + \bbR X_W $, of dimension 9. \\
The degrees of freedom are, at vertices two for the value of the 2-form, and at edges one for the integral.

\item $A^2(T) = \poly^0 \difo^2(T)$ consists of constant 2-forms on $T$, of dimension 1.\\
The degree of freedom is the integral.
\end{itemize}

These degrees of freedom provide a commuting interpolator. This complex is minimal, among complexes with this regularity, by the remarks that will be made in the last paragraph of \S \ref{sec:gfes}.

\begin{remark}\label{rem:gn14a}
In \cite{GuzNei14a} a complex of regularity $(2,1,0)$ equipped with commuting interpolators is also defined. Instead of resolving the Clough-Tocher element (full or reduced) like ours, their complex resolves a $\rmC^1$ element due to Zienkiewicz that contains rational functions. The dimensions of their three spaces is $(12,12,1)$, which is intermediate between our minimal complex, with dimensions $(9,9,1)$ and the previous complex, with dimensions $(12,20,9)$.

They also define high order versions of their complex.
\end{remark}

\begin{center}
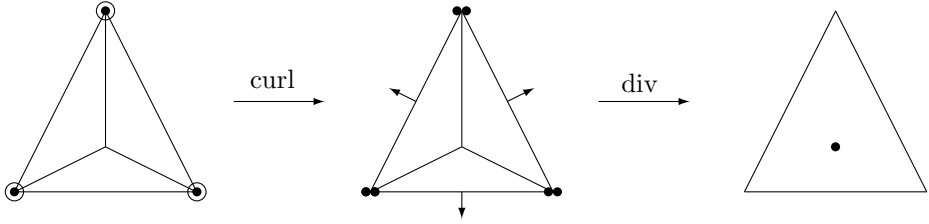
\begin{figure}
\setlength{\unitlength}{1.2cm}
\begin{picture}(2,2)(-0.8,0)
\put(0,0){
\begin{picture}(2,2)
\put(-1, 0){\line(1,2){1}} 
\put(0, 2){\line(1,-2){1}}
\put(-1,0){\line(1,0){2}}
\put(-1.,0){\circle*{0.1}}
\put(-1.,0){\circle{0.2}}
\put(1.,0){\circle*{0.1}}
\put(1.,0){\circle{0.2}}
\put(0,2){\circle*{0.1}}
\put(0,2){\circle{0.2}}
\put(0, 0.5){\line(0, 1){1.5}}
\put(0, 0.5){\line(2, -1){1}}
\put(0, 0.5){\line(-2, -1){1}}
\end{picture}
}

\put(1.5, 1){\vector(1, 0){1}}
\put(1.68, 1.15){$\curl$}

\put(4,0){
\put(-1, 0){\line(1,2){1}} 
\put(0, 2){\line(1,-2){1}}
\put(-1,0){\line(1,0){2}}
\put(-0.05,2){\circle*{0.1}}
\put(0.05,2){\circle*{0.1}}
\put(-1.05,0){\circle*{0.1}}
\put(-0.95,0){\circle*{0.1}}
\put(-1.05,0){\circle*{0.1}}
\put(-0.95,0){\circle*{0.1}}
\put(1.05,0){\circle*{0.1}}
\put(0.95,0){\circle*{0.1}}

\put(0, 0.5){\line(0, 1){1.5}}
\put(0, 0.5){\line(2, -1){1}}
\put(0, 0.5){\line(-2, -1){1}}

\put(0.5, 1){\vector(2,1){0.3}}
\put(-0.5, 1){\vector(-2,1){0.3}}
\put(0, 0){\vector(0,-1){0.3}}
}

\put(5.5, 1){\vector(1, 0){1}}
\put(5.75, 1.1){{$\div$}}

\put(8,0){
\begin{picture}(2,2)
\put(-1, 0){\line(1,2){1}} 
\put(0, 2){\line(1,-2){1}}
\put(-1,0){\line(1,0){2}}
\put(0, 0.5){\circle*{0.1}}

\end{picture}
}

\end{picture}
\caption{Minimal Clough-Tocher complex with discontinuous pressure.
}
\end{figure}
\end{center}

\section{Generalized Finite Element Systems\label{sec:gfes}}

\paragraph{Motivation for finite element systems.}
To study the examples of the preceding section, some general theorems make the task easier. Moreover, specifying the degrees of freedom a priori can be difficult when one wants to go to higher order polynomials.

If we are given spaces $A^k(T)$ of $k$-forms on a cell $T$, we can actually forget about degrees of freedom and just consider the spaces $A^k(T')$ obtained by restriction to the faces $T'$ of $T$, with the appropriate definition of restriction, adapted to a particular regularity. Two properties turn out to be sufficient, in order to get a nice finite element:
\begin{itemize}
\item $\dim A^k(T) = \sum_{T'\subcell T} \dim A^k_0(T')$. Here, as will be detailed below, $T' \subcell T$ signifies that $T'$ is a subcell of $T$ and $A^k_0(T')$ denotes the subset of $A^k(T')$ consisting of $k$-forms whose restrictions to boundary subcells of $T'$ are $0$.
\item The sequence $A^\bs(T')$ is exact on each subcell $T'$ of $T$, except at index $0$, where the cohomology group has dimensions $1$, essentially consisting of the constant functions.
\end{itemize}
When these properties are satisfied we will show that the sequences $A^\bs_0(T')$ are exact except at index $\dim T'$, where the cohomology group has dimension $1$. Then one can define a commuting interpolator by using the so-called harmonic degrees of freedom, described below.

In a cellular complex the spaces $A^k(T')$ defined on faces should be well defined, in the sense that if they are obtained as the spaces of restrictions from a cell $T$ (containing $T'$) to $T'$, then they should be independent of $T$.

\paragraph{Definitions related to finite element systems.}

Let $\calT$ be a cellular complex. If $T, T'$ are cells in $\calT$ we write $T' \subcell T$ to signify that $T'$ is a subcell of $T$ (we consider that $T$ is a subcell of $T$). Given two cells $T$ and $T'$ in $\calT$, their relative orientation is denoted $\orient(T,T')$. It is $0$ unless $T'$ is a codimension one subcell of $T$, in which case it is $\pm 1$. Cellular cochain complex is denoted $\calC^\bs(\calT)$. Its differential, also called the coboundary map, is denoted $\delta: \calC^k(\calT) \to \calC^{k+1}(\calT)$. Its matrix in the canonical basis is given by relative orientations.

All complexes considered in this paper are cochain complexes in the sense that the differential increases the index.

\begin{definition} A \emph{finite element system} on $\calT$ consists of the following data, which includes both spaces and operators:

\begin{itemize}
\item We suppose that for each $T \in \calT$, and each $k \in \bbZ$ we are given a vector space $A^k(T)$. For $k<0$ we suppose $A^k(T) = 0$.

\item For every $T \in \calT$ and $k \in \bbZ$, we have an operator $\diff^k_T: A^k(T)  \to A^{k+1}(T)$ called \emph{differential}. Often we will denote it just as $\diff$. We require $\diff^{k+1}_T \circ \diff^k_T = 0$. This makes $A^\bs(T)$ into a complex.
\item  Given $T, T'$ in $\calT$ with $T' \subcell T$ we suppose we have \emph{restriction} maps:
\begin{equation}
\rest^k_{T'T} : A^k(T) \to A^k(T'),
\end{equation}
subject to:
\begin{itemize}
\item $\rest^{k+1}_{T'T} \diff^k_T = \diff^k_{T'} \rest^k_{T'T}$.
\item $\rest^k_{T''T}  = \rest^k_{T''T'} \rest^k_{T'T}$. 
\end{itemize} 
This makes the family $A^\bs(T)$, for $T \in \calT$,  into an inverse system of complexes.
\item We suppose we have a map $\const_T: \bbR \to A^0(T)$. It mimicks inclusion of constant scalar functions. We require:
\begin{itemize}
\item For $T \in \calT$, $\diff^0_T \const_T = 0$.
\item If $T' \subcell T$ are cells in $\calT$, $\rest_{T'T} \const_T = \const_{T'}$.
\end{itemize}
\item For $T$ a $k$-dimensional cell in $\calT$ we suppose we have an evaluation map $\eval : A^k(T) \to \bbR$. It mimicks integration of $k$-forms on a $k$-cell.
We suppose that the following formula holds, for $u \in A^{k-1}(T)$:
\begin{equation}\label{eq:Stokes}
\eval_T \diff_T u = \sum_{T' \in \partial T} \orient(T, T') \eval_{T'} \rest_{T'T} u.
\end{equation}
It's an analogue of Stokes theorem on $T$.
\end{itemize}
\end{definition}

If $\calT'$ is a cellular subcomplex of $\calT$, the spaces $A^k(T)$ with $T\in \calT'$ constitute an inverse system. The inverse limits can be identified as:
\begin{align}
{\ula \lim}_{T\in \calT'} A^k(T) = \{ & (u_T)_{T \in \calT'} \in \bigoplus_{T \in \calT'} A^k(T) \ : \ T' \subcell T \Rightarrow u_{T'} = \rest_{T'T} u_T \}
\end{align}
In other words ${\ula \lim}_{T\in \calT'} A^k(T)$ consists of families $(u_T)_{T \in \calT'}$, such that for each cell $T \in \calT'$ (of all dimensions) $u_T \in A^k(T)$, and the family is stable under restrictions to subcells. One can consider that such a family is given by a choice of $u_T \in A^k(T)$ on top-dimensional cells $T \in \calT'$, together with their restrictions to subcells, provided that these are single-valued, i.e. the restrictions to a subcell are the same from all top-dimensional neighboring cells.

We notice that, if $T$ is a cell and $\subcells(T)$ denotes the cellular complex consisting of all the subcells of $T$ in $\calT$, then the restriction maps provide an isomorphism:
\begin{equation}
\rest: A^\bs(T) \to {\ula \lim}_{T'\in \subcells (T)} A ^\bs(T').
\end{equation}
For this reason it seems safe to use the notation:
\begin{equation}
 A^\bs(\calT') = {\ula \lim}_{T\in \calT'} A^\bs(T).
\end{equation}
This will be used in particular when $\calT'$ is the boundary of a cell $T\in \calT$. In that case $\partial T$ denotes the cellular complex consisting of the strict subcells of $T$, and $A^k(\partial T)$ can be interpreted as consisting of families of elements $u_{T'} \in A^k(T')$ for $T' \in \partial T$ that are single-valued along interfaces inside the boundary.

Another way of formulating (\ref{eq:Stokes}) is that for any cellular subcomplex $\calT'$, the evaluation provides a cochain morphism:
\begin{equation}
\eval: A^\bs(\calT')\to \calC^\bs(\calT').
\end{equation}
We will later provide conditions under which it induces isomorphisms on cohomology groups, which would be an analogue of de Rham's theorem.

We denote by $A^k_0(T)$ the kernel of the induced map $\rest^k: A^k(T) \to A^k(\partial T)$. We consider that the boundary of a point is empty, so that if $T$ is a point $A^k_0(T) = A^k(T)$.

\begin{definition}
We say that $A$ admits extensions on $T \in \calT$, if the restriction map induces a surjection:
\begin{equation}
\rest^k: A^k(T) \to A^k(\partial T),
\end{equation}
for each $k$. We say that $A$ admit extensions on $\calT$, if it admits extensions on each $T\in \calT$.
\end{definition}

This notion corresponds to that of \emph{flabby sheaves} (\emph{faisceaux flasques} in French \cite{God73}), due to the following.
\begin{proposition}
The FES $A$ admits extensions on $\calT$ if an only if, for any cellular complexes $\calT'', \cal T'$  such that $\calT'' \subseteq \calT' \subseteq \calT$,  the restriction  $A^\bs(\calT') \to A^\bs(\calT'')$ is onto.
\end{proposition}

In particular if $A$ admits extensions, then, when $T'$ is a subcell of $T$, the restriction $A^\bs(T) \to A^\bs(T')$ is onto. However this is in general a strictly weaker condition than the extension property. To see this, consider for instance the finite element spaces $A^0(T)$ consisting of $\poly^1$ functions on a quadrilateral $S$, on its edges $E$ and on its vertices $V$. Then the restriction from $A^0(S)$ to each edge $A^0(E)$ is onto, as are the other restrictions from faces to subfaces, but the restriction from $A^0(S)$ to $A^0(\partial S)$ is not onto, since the latter has dimension 4 but the former had dimension only 3.  

\begin{definition}
We say that $A^\bs$ is exact on $T$ when the following sequences are exact:
\begin{equation}\label{eq:exactd}
\xymatrix{
0 \ar[r]&  \bbR \ar[r]^\const & A^0(T) \ar[r]^\diff & A^1(T) \ar[r]^\diff & \ldots 
}
\end{equation}
We say that $A^\bs$ is locally exact on $\calT$ when $A^\bs$ is exact on each $T \in \calT$.
\end{definition}

\begin{definition}
We say that $A$ is \emph{compatible} when it admits extensions and is locally exact.
\end{definition}

\paragraph{de Rham type theorems.}

The following theorem extends Proposition 5.16 in \cite{ChrMunOwr11}:
\begin{theorem}
Suppose that the element system $A$ is compatible. Then the evaluation maps $\eval: A^\bs(\calT) \to \calC^\bs(\calT)$ induces isomorphisms on cohomology groups.
\end{theorem}
\begin{proof}
The proof of Proposition 5.16 in \cite{ChrMunOwr11} works verbatim.
\end{proof}  

We also have the following extension of Proposition 5.17 in \cite{ChrMunOwr11}:
\begin{theorem} \label{theo:altcomp} Suppose that $A$ has extensions. Then $A$ is compatible if and only if the following two conditions hold:
\begin{itemize}
\item For each $T\in \calT$ the (''inclusion of constants'') map $\const: \bbR \to A^0(T)$ is injective.
\item For each $T \in \calT$ the sequence $A^\bs_0(T)$ has nontrivial cohomology only at index $k = \dim T$, and there the induced map:
\begin{equation}
\eval: \rmH^k A^\bs_0(T) \to \bbR,
\end{equation}
is an isomorphism (it is well defined by (\ref{eq:Stokes})).
\end{itemize}

\end{theorem}

\begin{proof}

We suppose $m >0$ and that the equivalence has been proved for cellular complexes consisting of cells of dimension $n< m$.

Let $T\in \calT$ be a cell of dimension $m$. We suppose that $A$ is compatible on the boundary of $T$. Since the boundary is $(m-1)$ dimensional we may apply the above de Rham theorem there.

We write the following diagram:
\begin{equation}\label{eq:shortexact}
\xymatrix{
0 \ar[r] & A^\bs_0(T) \ar[r] \ar[d]^\eval & A^\bs(T) \ar[r] \ar[d]^\eval & A^\bs(\partial T) \ar[r] \ar[d]^\eval  & 0 \\
0 \ar[r] & \calC^\bs_0(T) \ar[r] & \calC^\bs(T) \ar[r] & \calC^\bs(\partial T) \ar[r]  & 0 
}
\end{equation}

On the rows, the second map is inclusion and the third arrow restriction. Both rows are short exact sequences of complexes. The vertical map is the de Rham map. The diagram commutes.

We write the two long exact sequences corresponding to the two rows, and connect them by the map induced by the de Rham map.
\begin{equation}
\xymatrix{
\rmH^{k-1} A^\bs(T)  \ar[r] \ar[d]  & \rmH^{k-1} A^\bs(\partial T)  \ar[r] \ar[d] &\rmH^{k}  A^\bs_0(T) \ar[r] \ar[d] & \rmH^{k} A^\bs(T)  \ar[r] \ar[d]  & \rmH^{k} A^\bs(\partial T) \ar[d]\\
\rmH^{k-1} \calC^\bs(T)  \ar[r]  & \rmH^{k-1} \calC^\bs(\partial T)  \ar[r] &\rmH^{k}  \calC^\bs_0(T) \ar[r]  & \rmH^{k} \calC^\bs(T)  \ar[r]  & \rmH^{k} \calC^\bs(\partial T) 
}
\end{equation}

Suppose that (\ref{eq:exactd}) is exact. Then the first and fourth vertical maps are isomorphisms. By the induction hypothesis the second and fifth are isomorphisms. By the five lemma, the third one is an isomorphism. This can be stated as announced.

Suppose that the two stated conditions hold. One applies the five lemma to the long exact sequence, and obtains that $A^\bs(T)$ is exact, except at index $0$. The cohomology group of index $0$ is one dimensional, and must consist of the constants, by injectivity of their inclusion.
\end{proof}

\paragraph{Extensions, dimension counts and harmonic interpolation.}
The following proposition almost exactly reproduces Proposition in \cite{ChrRap16}.

\begin{proposition}\label{prop:extrec} Suppose that $A^k$ is an element system and that $T \in \calT$. Suppose that, for each cell $U \in \partial T$, each element $v$ of $A^k_0(U)$ can be extended to an element $u $ of $A^k(T)$ in such a way that, $\rest_{UT} u = v$ and for each cell $U' \in \partial T$ with the same dimension as $U$, but different from $U$, we have $\rest_{U'T} u  = 0$. Then $A^k$ admits extensions on $T$.
\end{proposition}
\begin{proof}
In the situation described in the proposition we denote by $\ext_U v = u$ a chosen extension of $v$ (from $U$ to $T$).

Pick $v \in A^k(\partial T)$.  Define $u_{-1}= 0 \in A^k(T)$.

Pick $l \geq -1$ and suppose that we have a $u_{l} \in A^k(T)$ such that $v$  and $u_l$ have the same restrictions on all $l$-dimensional cells in $\partial T$. Put $w_l = v - \rest_{\partial T\, T} u_l \in A^k(\partial T)$. For each $(l+1)$-dimensional cell $U$ in $\partial T$, remark that $\rest_{U\partial T} w_l \in A^k_0(U)$, so we may extend it to the element $\ext_U \rest_{U\partial T} w_l \in A^k(T)$.
Then put:
\begin{equation}
u_{l+1} = u_l + \sum_{U \, : \, \dim U = l+1} \ext_U \rest_{U\partial T} w_l.
\end{equation}
Then $v$ and $u_{l+1}$ have the same restrictions on all $(l+1)$-dimensional cells in $\partial T$. 

We may repeat until $l+1 = \dim T$ and then $u_{l+1}$ is the required extension of $v$.
\end{proof}

\begin{proposition}\label{prop:extdim}
Let $A$ be a FES on a cellular complex $\calT$. Then: 
\begin{itemize}
\item We have:
\begin{equation}\label{eq:bounddim}
\dim A^k (\calT) \leq \sum_{ T\in \calT}\dim A^k_0(T).
\end{equation}
\item Equality holds in (\ref{eq:bounddim})  if and only if $A^k$ admits extensions on each $T \in \calT$.
\end{itemize}
\end{proposition}
\begin{proof}
The proof in \cite{ChrRap16} works verbatim.
\end{proof}

Suppose now that we are discretizing differential forms, say the sequence $\rmH^1_\rmd\difo^\bs(S)$ or, more precisely, $\rmC^0_\rmd\difo^k(S)$. For each cell $T$, equip each space of double traces of $\rmC^0_\rmd\difo^k(S)$, with a continuous scalar product $\langle \cdot | \cdot \rangle$, typically a variant of the $\rmL^2$ product on forms. For a given finite element system $A$ (equipped with double traces for the restrictions), define spaces $\calF^k(T)$ of degrees of freedom as follows. For  $k = \dim T$:
\begin{align} 
\calF^k(T) & = \{\langle \cdot |  v \rangle \ : \ v \in \ker \rmd | A^{k}_0(T)\} \oplus \{ \bbR \ts \int \cdot \}, \label{eq:harmdofk}
\end{align}
and for $k \neq \dim T$:
\begin{align}
\calF^k(T) & = \{\langle \cdot | v \rangle \ : \ v \in \ker \rmd | A^{k}_0(T)\} \oplus \{ \langle \rmd \cdot | v \rangle \ : \ v \in \rmd A^{k}_0(T)\}. \label{eq:harmdof}
\end{align} 
This is the natural generalization, to the adopted setting, of so-called projection based interpolation, as defined in \cite{DemBab03}\cite{DemBuf05}. We call these the \emph{harmonic degrees of freedom}.  For compatible finite element systems these degrees of freedom are unisolvent and yield a commuting interpolator $\rmC^0_\rmd\difo^\bs(S) \to A^\bs(\calT)$, which we call the \emph{harmonic interpolator}. This topic is detailed in \S 2.4 of \cite{ChrRap16}, see in  particular Proposition 2.8 of that paper.

\paragraph{Minimality.}

Consider a FES $A$ on a cellular complex $\calT$ where the topdimensional cells are domains in a fixed vectorspace $\bbV$ of dimension $n$. Suppose furthermore that each certex lies in an $n$-dimensional cell. Suppose that, when $T$ is a top-dimensional cell, $A^k(T)$ is a space of $k$-forms, containing the constant ones. If $A$ is compatible then in particular for each vertex $V \in \calT$, the restriction from $A^k(T)$ to $A^k(V)$ is onto. Depending on the nature of restriction we deduce:
\begin{itemize}
\item If restriction is the pullback, then $A^k(V) = 0$, except for $k= 0$, in which case it has dimension $1$.
\item If restriction is the trace, then $A^k(V) = \alter^k(\bbV)$.
\item If restriction is the double trace then $A^k(V) = \alter^k(\bbV) \oplus \alter^{k+1}(\bbV)$.
\end{itemize}
Moreover, if $A$ is compatible, then we must have, for any $k$-dimensional cell, $\dim A^k_0(T) \geq 1$, by Theorem \ref{theo:altcomp}.

These considerations provide a lower bound on $\dim A^k(\calT)$ in view of Proposition \ref{prop:extdim}. We will see examples where this lower bound is attained. These examples are then minimal FES.

This paper defines four minimal FES: two in 2D and two in 3D. In each dimension we distinguish between continuous and discontinuous divergence.

The topic of minimal FES is studied in more detail in \cite{ChrGil16}, in the case where the restriction is the pullback. General recipies for constructing small compatible FES within a big compatible FES are provided.

\section{High order composite elements in 2D\label{sec:twodhigh}}

\paragraph{Definition of finite element spaces.}

On a triangle $T$, we define a complex of regularity $(2, 1+, 1)$ depending on a parameter $p \geq 3$. The choice $p=3$ was described previously, in Proposition \ref{prop:lowone}, except for the characterization of spaces attached to faces.

We define the following spaces:
\begin{itemize}
\item $A^0(T) = \rmC^1\poly^p \difo^0(\calR)$.\\
It consists of the functions which are $\calR$-piecewise in $\poly^p$, and which are of class $\rmC^1(T)$. 

\item $A^1(T) = \rmC^0_\rmd\poly^{p-1}\difo^1(\calR)$.\\
It consists of  the 1-forms which are $\calR$-piecewise in $\poly^{p-1}$, and which are of class $\rmC^0(T)$ with exterior derivative in $\rmC^0(T)$. 

\item $A^2(T)= \rmC^0\poly^{p-2}\difo^2(\calR)$.\\
It consists of the 2-forms which are $\calR$-piecewise in $\poly^{p-2}$, and  which are of class $\rmC^0(T)$. 
\end{itemize}

We analyse this complex as follows. First we notice:
\begin{proposition}
The following sequence is exact:
\begin{equation}\label{eq:exct}
\xymatrix{
0 \ar[r] &\bbR \ar[r] &  A^0(T) \ar[r] & A^1(T) \ar[r] & A^2(T) \ar[r] & 0,
}
\end{equation}
The dimensions are:
\begin{align}
\dim A^0(T) & = (3/2) p (p-1) + 3,\\
\dim A^1(T) & = 3 p (p-2),\\
\dim A^2(T) & =(3/2)(p-1)(p-2) - 2.
\end{align}

\end{proposition}
\begin{proof}
\tpoint Exactness follows from an application of the Poincar\'e operator associated with the inpoint $W$.

\tpoint For $A^0(T)$, the dimension is given in \cite{DouDup79}.

\tpoint For $A^2(T)$ the space consists of continuous piecewise $\poly^r$ functions, on a mesh with 4 vertices, 6 edges and 3 triangles, so with $r = p-2$. Adding the dimensions of the bubblespaces we get:
\begin{equation}
\dim A^2(T) = 4 + 6 (r-1) + 3 (\frac{(r-1)(r-2)}{2}) = \frac{3}{2}r(r+1) +1 .
\end{equation}

\tpoint The dimension of $A^1(T)$ can then be deduced from the exactness of (\ref{eq:exct}):
\begin{equation}
\dim A^1(T) = -1 + \dim A^0(T) + A^2(T).
\end{equation}
This completes the proof.
\end{proof}

\begin{remark}
The dimensions are those one obtains by the perhaps na\"ive approach of counting constraints on piecewise-polynomial differential forms. 

For instance, for $A^0(T)$ one starts with the space of $\calR$-piecewise polynomials of degree $p$. It has dimension $(3/2)(p+2)(p+1)$. To be $\rmC^1$ at $W$ one imposes two equalities of first order jets, which amounts to 6 conditions. Then, on the edges joining $W$ to the vertices, one expresses continuity, knowing we already have continuity at $W$ as well as continuity of the directional derivative at $W$ along the edge. This gives  $3(p-1)$ conditions. Finally one expresses continuity of a transverse derivative on the interior edges, knowing that we already have continuity of it at $W$. This also gives $3(p-1)$ conditions. 

This gives a lower bound on the dimension, since we are not certain at this point that the imposed conditions are linearly independent.
\end{remark}

Having examined the spaces $A^k(T)$, we now look at what happens on the faces of $T$:

\case{Vertices.} We define, at a vertex $V$:
\begin{itemize}
\item $A^0(V) = \bbR \oplus \alter^1(\bbV)$ interpreted as a value and a value of the differential. Its dimension is 3.
\item $A^1(V) = \alter^1(\bbV) \oplus \alter^2(\bbV)$ interpreted as a value and a value of its exterior derivative. Its dimension is 3.
\item $A^2(V) = \alter^2(\bbV)$. Its dimension is 1.
\end{itemize}
Notice, in view of Lemma \ref{lem:exind}, that we have a well defined complex:
\begin{equation}
\xymatrix{
0 \ar[r] &\bbR \ar[r] &  A^0(V) \ar[r] & A^1(V) \ar[r] & A^2(V) \ar[r] & 0,
}
\end{equation}
where the second arrow $v_0\mapsto (v_0, 0)$, the third is $(v_0,v_1) \mapsto (v_1, 0)$ and the fourth one is $(v_0, v_1)\mapsto v_1$. Remark that the complex is exact.

\case{Edges.} At an edge $E$ we define:
\begin{itemize}
\item  $A^0(E)$ is the subspace of $\poly^{p}(E) \oplus \poly^{p-1}(E) \otimes \alter^1(\bbV)$ consisting of admissible pairs $(v_0, v_1)$. Its dimension is $p+1 + p= 2p+1$. 
\item $A^1(E) = \poly^{p-1}(E)\otimes\alter^1(\bbV) \oplus \poly^{p-2}(E) \otimes\alter^2(\bbV)$. Its dimension is $2p +p-1 = 3p -1$.
\item $A^2(E) = \poly^{p-2}(E)\otimes\alter^2(\bbV)$. Its dimension is $p-1$.
\end{itemize}
Again, in view of Lemma \ref{lem:exind},  we notice that we have a well defined complex:
\begin{equation}
\xymatrix{
0 \ar[r] &\bbR \ar[r] &  A^0(E) \ar[r] & A^1(E) \ar[r] & A^2(E) \ar[r] & 0,
}
\end{equation}
and that it is exact.

We remark that $A$ defines a finite element system on $\subcells (T)$, with respect to restriction operators defined by taking double-traces. The crucial missing point is the extension property (flabbyness).

The following result is immediate.
\begin{proposition}
For any edge $E$ of $T$, $A$ admits extensions from $\partial E$ to $E$. Moreover:
\begin{align}
\dim A^0_0(E) & = 2p +1 - 2\cdot 3 = 2p -5,\\ 
\dim A^1_0(E) & = 3p -1 - 2\cdot 3 = 3p - 7,\\
\dim A^2_0(E) & = \phantom{1} p-1 - 2\cdot 1  = \phantom{1} p-3.
\end{align}
And there is nontrivial cohomology only at index $k=1$, where it has dimension $1$.
\end{proposition}

\begin{theorem}
The Finite Element System $A$ admits extensions from $\partial T$ to $T$. Hence it is compatible.
\end{theorem}
\begin{proof}
We use Proposition \ref{prop:extrec}. What is required is to prove some extension properties from vertices to edges and triangles, and from edges to triangles. These required properties are proved in the two next paragraphs.
\end{proof}

We use the term \emph{jet} informally. An $r$-jet corresponds to a Taylor expansion of order $r$ in some vector bundle, which will here be a vector bundle of differential forms. However for the highest order partial derivatives, only a certain combination of them, corresponding to the exterior derivative, will be used. Moreover the jet exists even when a section it should be the expansion of, is not known a priori.

\paragraph{Extension of 1-jets from vertices.}

In this section we consider elements in dimension $2$ but our construction of extension from vertices is valid in any dimension. Let then $\bbV$ be a vector space of finite dimension and let $V$ be a point in $\bbV$.

We are interested in complexes at $V$ of the form:
\begin{equation}
A^k(V) = \alter^k(\bbV) \oplus \alter^{k+1}(\bbV).
\end{equation}

Suppose we are given $(v_0, v_1) \in \alter^k(\bbV) \oplus \alter^{k+1}(\bbV)$ at vertex $V$.  Suppose $T$ is a simplex, of arbitrary dimension, containing $V$. We want to find a $k$-form $u_0$ on $T$ whose double trace is $(v_0, v_1)$. In other words we want an admissible pair $(u_0, u_1)$ whose traces are $(v_0, v_1)$. 

Let $\lambda$ be the barycentric coordinate on $T$ with respect to vertex $V$, and let $X$ be the canonical vectorfield $X: x \mapsto x - V$. Notice that for any $w \in \alter^{k+1}(\bbV)$ considered as a constant $(k+1)$-form, we have $\rmd (w \ctr X) = (k+1) w$.

The admissible pair $(\lambda^2 v_0, 2\lambda \rmd \lambda \wedge v_0)$ on $T$ restricts to $(v_0, 2 \rmd \lambda \wedge v_0)$ at $V$. We therefore put $w_1 = v_1 -  2\rmd \lambda \wedge v_0$, and we want to find an extension of $(0, w_1)$. We notice that the following pair on $T$ is both admissible and restricts to $(0, w_1)$ at $V$:
\begin{equation}
(\frac{1}{k+1} \lambda^2 w_1 \ctr X, \frac{2}{k+1}\lambda \rmd \lambda \wedge (w_1 \ctr X) + \lambda^2 w_1)
\end{equation}

All in all, we extend the data at $V$ to $T$ by the formula: 
\begin{align}
(u_0, u_1)  = &(\lambda^2 v_0, 2\lambda \rmd \lambda \wedge v_0) +\\
&  (\frac{1}{k+1} \lambda^2 w_1 \ctr X, \frac{2}{k+1}\lambda \rmd \lambda \wedge (w_1 \ctr X) + \lambda^2 w_1). 
\end{align}
Notice that $u_0$ is a differential $k$-form of polynomial degree 3, that $u_1$ is a $(k+1)$-form of degree 2, that the pair $(u_0,u_1)$ is admissible, and that its restriction to the other vertices of $T$ is $0$, in the sense of double-traces. More stongly, the restriction to the face opposite to $V$ in $T$ is $0$. This construction can be used to obtain basisvectors attached to the vertices of the global spaces.

\begin{remark}
If $T$ is a triangle, Proposition \ref{prop:lowone} guarantees that we have extensions from the vertices of $T$ to $T$, as required in Proposition \ref{prop:extrec}, simply by matching degrees of freedom.
\end{remark}

\paragraph{Extension of polynomial 1-jets from edges to triangles.}
Now suppose $E$ is an edge of a triangle $T$, living in a vector space $\bbV$ of dimension $2$. We wish to extend data on $E$ to $T$, so as to be able to apply Proposition \ref{prop:extrec} .

Fix $p$ such that $ p \geq 3$. We consider the following spaces, for $0 \leq k \leq \dim \bbV$.
\begin{align}
A^k(E) = \{ & (v_0, v_1) \in \poly^{p-k}(E) \otimes \alter^k(\bbV) \oplus \poly^{p-k-1}(E) \otimes \alter^{k+1}(\bbV) \ : \ \nonumber \\
 & (v_0, v_1) \textrm{ is admissible} \}.
\end{align}
The admissibility condition is non-trivial only for $k=0$.

We label the vertices of $E$ with $0$ and $1$, and the third vertex of $T$ is labelled with $2$. The barycentric coordinates on $T$ are, accordingly, denoted $\lambda_0, \lambda_1, \lambda_2$.

We suppose we have chosen an inpoint $W$ on $T$, and we divide $T$ into three triangles by joining $W$ to the vertices of $T$. The simplicial complex so obtained is denoted $\calR$.

\begin{lemma}
There is a function $\Phi \in \rmC^1\poly^3(\calR)$ such that:
\begin{align}
\trace_{\partial T}(\Phi) &= 0,\\
\trace_{\partial T} (\rmd \Phi) & = \trace_{\partial T} (\lambda_0 \lambda_1 \rmd \lambda_2). 
\end{align}
\end{lemma}
\begin{proof}
Follows from Proposition \ref{prop:lowone} by matching degrees of freedom.
\end{proof}

\begin{lemma}
There is a 1-form $\Psi \in \rmC^0_\rmd\poly^2\difo^1(\calR)$ such that:
\begin{align}
\trace_{\partial T}(\Psi) &= \trace_{\partial T} (\lambda_0\lambda_1 \rmd \lambda_1),\\
\trace_{\partial T} (\rmd \Psi) & =0 . 
\end{align}
\end{lemma}
\begin{proof}
Follows from Proposition \ref{prop:lowone} by matching degrees of freedom.
\end{proof}

Suppose we are given $(v_0,v_1) \in A^k_0(E)$ and that we wish to extend it to $T$. We may extend this data by $0$ to all of $\partial T$.

\case{Case $k=0$.} First we remark that $v_0$ is of the form:
\begin{equation}
v_0 = w_0(\lambda_1) \lambda_0^2 \lambda_1^2 ,
\end{equation}
where $w_0 \in \poly^{p-4}(E)$. In this form $v_0$ is trivially extendable to $T$, as a function $u_0 \in \poly^p(T)$. Substracting $(\trace_E u_0, \trace_E \rmd u_0)$ from $(v_0,v_1)$ leaves us with data where $v_0 = 0$. Assuming now that $v_0 = 0$, admissibility shows that $v_1$ is of the form:
\begin{equation}
v_1 = w_1(\lambda_1) \lambda_0 \lambda_1  \rmd \lambda_2,
\end{equation}
where $w_1 \in \poly^{p-3}(E)$. Then we extend $(0, v_1)$ to $T$ as the admissible pair:
\begin{align}
(u_0, u_1) & = (u_0, \rmd u_0),\\
& = ( w_1(\lambda_1) \Phi,  \dot w_1(\lambda_1) \rmd \lambda_1  \Phi + w_1(\lambda_1) \rmd \Phi).
\end{align}
In our setup it is only $u_0$ which is of interest on $T$, but we need the traces of both $u_0$ and $\rmd u_0$ on $\partial T$.

Notice also that the constructed extension satisfies $u_0 \in \rmC^1\poly^p\difo^0(\calR)$.

\case{Case $k=1$.} We remark that the data is of the form:
\begin{align}
v_0 &= w_0(\lambda_1) \lambda_0 \lambda_1 \rmd \lambda_1 + w_1(\lambda_1) \lambda_0 \lambda_1 \rmd \lambda_2,\\
v_1 &= w_2(\lambda_1) \lambda_0 \lambda_1 \rmd \lambda_1 \wedge \rmd \lambda_2.
\end{align}
With $w_0 \in \poly^{p-3}(E)$, $w_1 \in \poly^{p-3}(E)$ and $w_2 \in \poly^{p-4}(E)$. We essentially extend the three different components separately, but in a precise order.

First, let $\tilde w_2 \in \poly^{p-3}(E)$ denote an antiderivative of $w_2$. Put:
\begin{equation}
u_0 = \tilde w_2(\lambda_1) \rmd \Phi \in \rmC^0_\rmd \poly^{p-1} \difo^1(\calR).
\end{equation}
Then:
\begin{equation}
\rmd u_0 = w_2(\lambda_1) \rmd \lambda_1 \wedge \rmd \Phi,
\end{equation}
whose trace is $v_1$. This leaves us with the problem of extending data where $w_2 = 0$.

Second, define:
\begin{equation}
u_0 = w_1(\lambda_1) \rmd \Phi + \dot w_1(\lambda) \rmd \lambda_1 \Phi \in \rmC^0_\rmd \poly^{p-1} \difo^1(\calR).
\end{equation}
Then $\rmd u_0 = 0$, so in particular $\trace_{\partial T} u_0 = 0$. Moreover:
\begin{equation}
\trace_{\partial T} u_0 =  w_1(\lambda_1) \lambda_0 \lambda_1 \rmd \lambda_2.
\end{equation}
This leaves us with the problem of extending data where $w_2 = 0$ and $w_1 =0$.

Third, define:
\begin{equation}
u_0 = w_0(\lambda_1) \Psi  \in \rmC^0_\rmd \poly^{p-1} \difo^1(\calR).
\end{equation}
Then:
\begin{equation}
\trace_{\partial T} u_0 = w_0(\lambda_1) \lambda_0 \lambda_1 \rmd \lambda_1.
\end{equation}
and moreover:
\begin{align}
\trace_{\partial T}(\rmd u_0) & = \trace_{\partial T} (\dot w_0(\lambda_1) \rmd \lambda_1 \wedge \Psi) + \trace_{\partial T} (w_0(\lambda_1) \rmd \Psi),\\
& = 0.
\end{align}
This completes the extension procedure.

\case{Case $k=2$.} Then $v_1=0$ and $v_0$ is of the form:
\begin{equation}
v_0 = w_0(\lambda_1)\lambda_0\lambda_1 \rmd \lambda_0 \wedge \rmd \lambda_1,
\end{equation}
for some $w_0 \in \poly^{p-4}(E)$. We extend $v_0$ to $T$ as:
\begin{equation}
u_0 = w_0(\lambda_1) \rmd \lambda_0 \wedge \Psi \in \rmC^0 \poly^{p-2} \difo^2(\calR).
\end{equation}

\section{Tools for composite finite elements\label{sec:toolhigh}}

We develop some tools that will be used to define finite element sequences in dimension $n \geq 3$.

\paragraph{Various refinements of simplices.}
A simplex is a finite non-empty set. Its subsimplices are the non-empty subsets. The geometric realization of a simplex $T$ in a vector space containing the vertices, is its convex hull, denoted $|T|$. Geometric realizations are examples of cells. If $T$ is a simplex with vertices $V_0, \ldots , V_k$ we also write $T = [V_0, \ldots, V_k]$. 

If $T$ is a cell in a cellular complex $\calT$, we denote by $\subcells_\calT(T)$ the set of subcells of $T$ in $\calT$, which is also a cellular complex. We denote by $\subcells^k_\calT(T)$ the set of those subcells of $T$ which have dimension $k$. When no confusion is possible we omit the subscript $\calT$. In particular, if $T$ is a simplex the associated simplicial complex is denoted $\subcells(T)$.

For each simplex $T$ we choose an interior point $W_T$, called the inpoint of $T$.

\begin{definition} Given a simplex $T$ we denote by $\refi_m(T)$ the simplicial complex consisting of simplices of the form:
\begin{equation}
[W_{T_k}, W_{T_{k-1}}, \ldots, W_{T_0}, V_{0}, \ldots, V_l],
\end{equation}
such that:
\begin{itemize}
\item $T' = [V_0, \ldots, V_l]$ is a subsimplex of $T$ of dimension $l \leq m$,
\item $T_0 , \ldots , T_{k-1} , T_k$ are subsimplices of $T$ of dimension at least $m+1$,
\item The simplices are nested as follows, with strict inclusions: 
\begin{equation}
T' \subcellstrict T_0 \subcellstrict \ldots \subcellstrict T_{k-1} \subcellstrict T_k.
\end{equation}
\end{itemize}
We call $\refi_m(T)$ the  \emph{$m$-refinement} of $T$.
\end{definition}

In particular $\refi_0(T)$ is the barycentric refinement of $T$, at least when the inpoints are chosen to be the isobarycenters. We  see that $\refi_m(T)$ only uses inpoints of subsimplices of $T$ of dimension at least $m+1$ ; subsimplices of $T$ of dimension at most $m$ are not refined. Another way of saying this is that $\subcells (T)$ and $\refi_m(T)$ have the same $m$-skeleton (the $m$-skeleton of a cellular complex is the cellular complex consisting of those cells that have dimension at most $m$).  For $m \geq \dim T$ we have $\refi_m(T) = \subcells(T)$.

When choosing the inpoints, one is interested in satisfying special properties for adjacent simplices in some simplicial complex, as reviewed in \cite{LaiSch07}:
\begin{itemize}
\item In dimension 2, $\refi_1(T)$ is known as a Clough-Tocher split. One is also interested in splits where the inpoints of edges lie on the lines joining the inpoints of the adjacent triangles. Then $\refi_0(T)$ is known as a Powell-Sabin split. 

\item In dimension 3, one is interested in splits where the inpoints on faces lie on the lines joining the inpoints of the two adjacent tetrahedra. Then $\refi_1(T)$ is known as a Worsey-Farin split, after \cite{WorFar87}. If, in addition, the inpoint on edges lie on a plane cointaining all the inpoints of the adjacent tetrahedra (i.e. those containing the edge), then $\refi_0(T)$ is called a Worsey-Piper split, after \cite{WorPip88}.

\item Actually \cite{WorFar87} defines $\refi_{1}(T)$ in arbitrary dimension $n$ and refer to it as generalized Clough-Tocher split. On  $n$-dimensional simplices one chooses arbitrary inpoints. On $(n-1)$-dimensional simplices the inpoint is the intersection point with the line joining the inpoints of the two adjacent topdimensional simplices. 

%On lower dimensional simplices the inpoint is again arbitrary (?).

\item Worsey-Piper splits may be difficult to construct. One example would be to choose, as inpoints, the circumcenters of all subsimplices. A sufficient condition for this choice to yield points in the interior of the simplices, is that simplices are strictly acute.  This is quite restrictive.

\item We also note that for $m = \dim T -1$, $\refi_m(T)$, which consists in adding the single inpoint $W_T$ to $T$ and cone it with the boundary simplices of $T$, is known as the Alfeld split of $T$, at least when $\dim T = 3$, see \cite{Alf84}.
\end{itemize}

The different types of refinements of a tetrahedron are illustrated in Figure \ref{fig:ref}. Not all subsimplices are represented, just those corresponding to one face of the tetrahedron.

We note the following:
\begin{lemma} We have:
\begin{itemize}
\item For any $m$, $\refi_m(T)$ is a refinement of $\refi_{m+1}(T)$.
\item If $U$ is a subsimplex of $T$ then:
\begin{equation}
\refi_m(U) = \{T' \in \refi_m(T) \ : \ |T'| \subseteq |U| \}.
\end{equation}
\end{itemize}
\end{lemma}

{
\begin{center}
\begin{figure}
\setlength{\unitlength}{1.2cm}
\begin{picture}(2,2)(-0.5,0)

\put(0,0){
\begin{picture}(2,2)

{\color{red}\polygon*(-0.1, 0.6)(0, 0.9)(0, 1.5)
%{\color{lightorange}\polygon*(-0.1, 0.6)(0, 1.5)(0.368, 0.748)}
\color{lightorange}\polygon*(-0.1, 0.6)(0.368, 0.748)(0, 0.9)
\color{yellow}\polygon*(0, 0.9)(0, 1.5)(0.368, 0.748)
}

\put(-0.5, 0.5){\circle*{0.1}}
%\put(-0.5, 0.5){\circle{0.2}}
\put(1, 0.5){\circle*{0.1}}
%\put(1, 0.5){\circle{0.2}}
\put(0, 1.5){\circle*{0.1}}
%\put(0, 1.5){\circle{0.2}}
\put(0, 0){\circle*{0.1}}
%\put(0, 0){\circle{0.2}}
\put(0, 0){\line(-1, 1){0.5}}
\put(0, 0){\line(2, 1){1}}
\put(1, 0.5){\line(-1, 1){1}}
\put(-0.5, 0.5){\line(1, 2){0.5}}
\put(0, 0){\line(0, 1){1.5}}
\put(-0.1, 0.6){\circle*{0.05}}

\multiput(-0.5,0.5)(0.05,0){30}{\circle*{0.03}}

\put(0, 0){\line(1, 2){0.5}}
\put(0, 1.5){\line(1, -2){0.6}}
\put(1, 0.5){\line(-5, 2){1}}
\multiput(-0.1, 0.6)(0.05,0.015){10}{\circle*{0.02}}
\multiput(-0.1, 0.6)(0.05,-0.021){15}{\circle*{0.02}}
\multiput(-0.1, 0.6)(0.022,0.06){6}{\circle*{0.02}}
\multiput(-0.1, 0.6)(0.005, 0.05){18}{\circle*{0.02}}
\multiput(-0.1, 0.6)(0.008, -0.05){12}{\circle*{0.02}}
\multiput(-0.1, 0.6)(0.05, -0.004){20}{\circle*{0.02}}
\multiput(-0.1, 0.6)(0.048, 0.033){13}{\circle*{0.02}}
\end{picture}
}

\put(3,0){
\begin{picture}(2,2)

%{\color{orange}\polygon*(-0.1, 0.6)(0, 1.5)(0.368, 0.748)(0, 0)}

{
%{\color{lightorange}\polygon*(-0.1, 0.6)(0, 1.5)(0.368, 0.748)}
\color{yellow}\polygon*(0, 0)(0, 1.5)(0.368, 0.748)
\color{red}\polygon*(0, 0)(0, 1.5)(-0.1, 0.6)
}

\put(-0.5, 0.5){\circle*{0.1}}
%\put(-0.5, 0.5){\circle{0.2}}
\put(1, 0.5){\circle*{0.1}}
%\put(1, 0.5){\circle{0.2}}
\put(0, 1.5){\circle*{0.1}}
%\put(0, 1.5){\circle{0.2}}
\put(0, 0){\circle*{0.1}}
%\put(0, 0){\circle{0.2}}
\put(0, 0){\line(-1, 1){0.5}}
\put(0, 0){\line(2, 1){1}}

\put(1, 0.5){\line(-1, 1){1}}
\put(-0.5, 0.5){\line(1, 2){0.5}}
\put(0, 0){\line(0, 1){1.5}}
\put(-0.1, 0.6){\circle*{0.05}}

\multiput(-0.5,0.5)(0.05,0){30}{\circle*{0.03}}

\put(0, 0){\line(1, 2){0.38}}
\put(0, 1.5){\line(1, -2){0.38}}
\put(1, 0.5){\line(-5, 2){0.62}}
\multiput(-0.1, 0.6)(0.05,0.015){10}{\circle*{0.02}}
%\multiput(-0.1, 0.6)(0.05,-0.021){15}{\circle*{0.02}}
%\multiput(-0.1, 0.6)(0.022,0.06){6}{\circle*{0.02}}
\multiput(-0.1, 0.6)(0.005, 0.05){18}{\circle*{0.02}}
\multiput(-0.1, 0.6)(0.008, -0.05){12}{\circle*{0.02}}
\multiput(-0.1, 0.6)(0.05, -0.004){20}{\circle*{0.02}}
%\multiput(-0.1, 0.6)(0.048, 0.033){13}{\circle*{0.02}}
\end{picture}

}

\put(6,0){
\begin{picture}(2,2)

{
%{\color{lightorange}\polygon*(-0.1, 0.6)(0, 1.5)(0.368, 0.748)}
\color{yellow}\polygon*(0, 0)(0, 1.5)(1, 0.5)
\color{red}\polygon*(0, 0)(0, 1.5)(-0.1, 0.6)
}

\put(-0.5, 0.5){\circle*{0.1}}
%\put(-0.5, 0.5){\circle{0.2}}
\put(1, 0.5){\circle*{0.1}}
%\put(1, 0.5){\circle{0.2}}
\put(0, 1.5){\circle*{0.1}}
%\put(0, 1.5){\circle{0.2}}
\put(0, 0){\circle*{0.1}}
%\put(0, 0){\circle{0.2}}
\put(0, 0){\line(-1, 1){0.5}}
\put(0, 0){\line(2, 1){1}}
\put(1, 0.5){\line(-1, 1){1}}
\put(-0.5, 0.5){\line(1, 2){0.5}}
\put(0, 0){\line(0, 1){1.5}}
\put(-0.1, 0.6){\circle*{0.05}}

\multiput(-0.5,0.5)(0.05,0){30}{\circle*{0.03}}

%\multiput(-0.1, 0.6)(0.05,0.015){10}{\circle*{0.02}}
%\multiput(-0.1, 0.6)(0.05,-0.021){15}{\circle*{0.02}}
%\multiput(-0.1, 0.6)(0.022,0.06){6}{\circle*{0.02}}
\multiput(-0.1, 0.6)(0.005, 0.05){18}{\circle*{0.02}}
\multiput(-0.1, 0.6)(0.008, -0.05){12}{\circle*{0.02}}
\multiput(-0.1, 0.6)(0.05, -0.004){20}{\circle*{0.02}}
%\multiput(-0.1, 0.6)(0.048, 0.033){13}{\circle*{0.02}}
\end{picture}

}
\put(9,0){
\begin{picture}(2,2)

{\color{yellow}\polygon*(0, 1.5)(1, 0.5)(0, 0)
\color{darkorange}\polygon*(-0.5, 0.5)(0, 1.5)(0, 0)
}

\put(-0.5, 0.5){\circle*{0.1}}
%\put(-0.5, 0.5){\circle{0.2}}
\put(1, 0.5){\circle*{0.1}}
%\put(1, 0.5){\circle{0.2}}
\put(0, 1.5){\circle*{0.1}}
%\put(0, 1.5){\circle{0.2}}
\put(0, 0){\circle*{0.1}}
%\put(0, 0){\circle{0.2}}
\put(0, 0){\line(-1, 1){0.5}}
\put(0, 0){\line(2, 1){1}}
\put(1, 0.5){\line(-1, 1){1}}
\put(-0.5, 0.5){\line(1, 2){0.5}}
\put(0, 0){\line(0, 1){1.5}}
%\put(-0.1, 0.6){\circle*{0.05}}

\multiput(-0.5,0.5)(0.05,0){30}{\circle*{0.03}}

%\multiput(-0.1, 0.6)(0.05,0.015){10}{\circle*{0.02}}
%\multiput(-0.1, 0.6)(0.05,-0.021){15}{\circle*{0.02}}
%\multiput(-0.1, 0.6)(0.022,0.06){6}{\circle*{0.02}}
%\multiput(-0.1, 0.6)(0.005, 0.05){18}{\circle*{0.02}}
%\multiput(-0.1, 0.6)(0.008, -0.05){12}{\circle*{0.02}}
%\multiput(-0.1, 0.6)(0.05, -0.004){20}{\circle*{0.02}}
%\multiput(-0.1, 0.6)(0.048, 0.033){13}{\circle*{0.02}}

\end{picture}

}

\end{picture}
\caption{\label{fig:ref}Refinements of a tetrahedron relative to one face:\newline
$\calR_0$ (Worsey-Piper), $\calR_1$ (Worsey-Farin), $\calR_2$ (Alfeld), $\calR_3$ (no split). }
\end{figure}
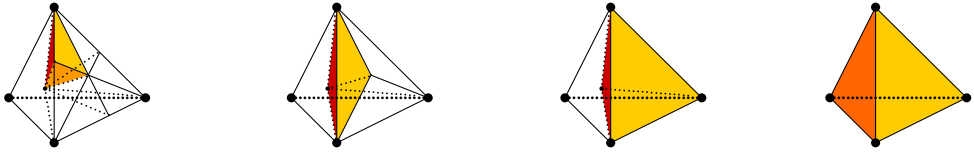
\end{center}
}

\paragraph{Alignments in meshes.}
Already on a triangular mesh in dimension $2$, continuity requirements involving derivatives, enforced on piecewise polynomials, may produce complicated spaces. The dimension will in general depend for instance on alignments of edges arriving at vertices. The following result pertains to one such situation.

Suppose $S$ is a two-dimensional vectorspace with a basis $(e_1, e_2)$. The basis vectors divide $S$ into four sectors, as follows. For the four possibilities of choices of signs $a,b \in\{+, -\}$ we consider the sectors: 
\begin{equation}
T_{ab} = \{ x_1 e_1 + x_2 e_2 \ : \ a x_1 \geq 0 \textrm{ and } b x_2 \geq 0 \}.
\end{equation}
We consider differential forms, which are piecewise polynomials with respect to this subdivision, with various continuity requirements across interfaces.

\begin{proposition}\label{prop:alired}
We have an exact sequence on $S$:
\begin{equation}
\xymatrix{
0 \ar[r] & \bbR \ar[r] & \rmC^1 \poly^2 \difo^0 \ar[r] & \rmC^0 \poly^1 \difo^1 \ar[r]  & \poly^0 \difo^2  \ar[r] & \bbR \ar[r] & 0
}
\end{equation}
where, more precisely:
\begin{itemize}
\item $\rmC^1 \poly^2 \difo^0$, the space of continuously differentiable piecewise polynomials of degree $2$, has dimension $8$. The arrow arriving from $\bbR$ is inclusion of constants. Any element $u$ will be uniquely determined by the values of the following data:
\begin{itemize}
\item the 1-jet at $0$, consisting of the function value $u(0)$ and the differential $\rmD u(0)$.

\item the directional second order derivatives at $0$, in the four directions $\pm e_1$ and $\pm e_2$,  which, by the way, are well defined.

\item the value of the second order derivative $\partial_1 \partial_2 u$, which, it turns out, must be the same in the four sectors.
\end{itemize}
\item $\rmC^0 \poly^1 \difo^1$ has dimension 10. The arrows arriving to and from this space are exterior derivatives.
\item $\poly^0 \difo^2$ has dimension $4$. The arrow to $\bbR$ is the following map:
\begin{equation}
u \mapsto u(++) - u(-+) + u(--) - u(+-).
\end{equation}
Here $u(ab)$ stands for the value of the two-form $u$ on $T_{ab}$, or more precisely $u[ae_1 + be_2](e_1, e_2)$.
\end{itemize}
\end{proposition}

\begin{remark}
It seems that, if we have just four sectors, without alignments of the edges then the sequence:
\begin{equation}
\xymatrix{
0 \ar[r] & \bbR \ar[r] & \rmC^1 \poly^2 \difo^0 \ar[r] & \rmC^0 \poly^1 \difo^1 \ar[r]  & \poly^0 \difo^2  \ar[r]  & 0
}
\end{equation}
is exact and $\rmC^1\poly^2\difo^0$ has dimension only $7$.

The situation is reminiscent of \cite{ScoVog85}, which is interested in the last part of the complex, for polynomials of higher order.
\end{remark}
 
\paragraph{Some spaces of piecewise polynomials on simplexes.}

We first recall:
\begin{proposition}
Suppose $T=[V_0, \ldots, V_n]$ is an oriented simplex of dimension $n$. 

\noindent Suppose that $u$ is a constant $n$-form on $T$. Then:
\begin{equation}
\int_T u = \frac{1}{n!} u(V_1-V_0, V_2 - V_0, \ldots,  V_n - V_0).
\end{equation}
Suppose that $u$ is affine on $T$ and $0$ at the vertices $V_1, \dots, V_n$. Then:
\begin{equation}
\int_T u = \frac{1}{(n+1)!} u[V_0](V_1-V_0, V_2 - V_0, \ldots,  V_n - V_0).
\end{equation}
\end{proposition}

Let $S$ be a simplex of dimension $n$. All faces $T$ of $S$ are supposed equipped with an orientation and a chosen inpoint $W_T$.

We shall prove some results of which the following constitute a first case:
\begin{proposition} We have the following:
\begin{itemize}
\item Suppose $u \in \rmC^1 \poly^2 \difo^0(\refi_{0} (S))$, that $\rmd u $ is $0$ at the vertices of $S$, and that $u$ has the same value at all vertices of $S$.  Then $u$ is constant on $S$. 

\item Suppose $u \in \rmC^0 \poly^1 \difo^1(\refi_{0} (S))$ and that $\rmd u= 0$. If $u$ is $0$ at the vertices of $S$ and the pullback of $u$ to $1$-dimensional faces of $S$ has integral $0$,  then $u = 0$. 
\end{itemize}

\end{proposition}
\begin{proof}
By induction on $\dim S$. For $\dim S =0$ there is nothing to prove. Supposing now $n\geq 1$ and that the result has been proved for simplexes $S$ with $\dim S < n$ we proceed as follows, supposing $\dim S = n$.

\tpoint Choose $u \in \rmC^1 \poly^2 \difo^0(\refi_{0} (S))$ and suppose that $\rmd u = 0$ at the vertices. On any $(n-1)$-face of $S$ the pullback of $u$ is constant by the induction step. Hence $u$ is constant on $\partial T$. Substracting this constant, we may suppose that $\trace_{\partial T} u = 0$. Let $\lambda_S$ be the barycentric coordinate on $S$ attached to the inpoint, so that $\lambda_S \in \rmC^0 \poly^1 \difo^0(\refi_{n-1}(S))$. We can write $u = \lambda_S v$ for some $v \in \rmC^0 \poly^1 \difo^0(\refi_{0} (S))$. The condition that $u \in \rmC^1(S)$ then gives $v \in \rmC^0 \poly^1 \difo^0(\refi_{n-1} (S))$. We write $\rmd u = \lambda_S \rmd v + v \rmd \lambda_S$ and deduce that $v$ is zero at the vertices of $S$. Hence $v$ is proportional to $\lambda_S$ : $v = c\lambda_S$. We get $\rmd u = 2 c \lambda_S \rmd \lambda_S$.  Since $\rmd \lambda_S$ is discontinuous at the inpoint of $S$, we deduce that $c = 0$, hence $u = 0$.

\tpoint Choose $u \in \rmC^0 \poly^1 \difo^1(\refi_{0} (S))$ such that $\rmd u = 0$ on $S$,  $u$ is $0$ at the vertices of $S$ and the pullback of $u$ to $1$-dimensional faces of $S$ has integral $0$. Write $u = \rmd v$ with $v \in \rmC^1 \poly^2 \difo^0(\refi_{0} (S))$. We have that $\rmd v$ is zero at vertices. Moreover $v$ has the same values at all vertices, by the one-dimensional Stokes. By the preceding result $v$ is constant, so $u = 0$.  
\end{proof}

The purpose of the next three propositions is to extend these results to $k$-forms for higher $k$. Eventually we want to show that if certain degrees of freedom are $0$ then the $k$-form is $0$.

Our first result is of the type that if certain degrees of freedom are $0$ then the $k$-form is $0$ at the center of the simplex.

\begin{proposition}\label{prop:inpointzero}Le $S$ be a simplex with $\dim S \geq 1$. Choose $k \geq 1$.
Suppose $u \in \rmC^0 \poly^1 \difo^k(\refi_{k-1} (S))$ and that $\rmd u = 0$. If $u$ is $0$ at the vertices of $S$ and the pullback of $u$ to $k$-dimensional faces of $S$ has integral $0$,  then $u(W_S) = 0$. 
\end{proposition}

\begin{proof}
\tpoint If $k=1$ the result was proved in the preceding two propositions. 

We suppose now that $k\geq 2$. The strategy is to prove that the pullback of $u$ to the $k$-simplices joining $W_S$ to the $(k-1)$ faces of $S$ is zero. The integrals of these pullbacks constitute a $(k-1)$-cochain on $S$, and we show that its coboundary is zero and that its (weighted) boundary is also $0$.

\tpoint For any $(k-1)$-face $U$ of $S$ define (the real number):
\begin{equation}
c_U = \int_{[W_S, U]} u.
\end{equation}
This defines a cochain $c_\bs \in \calC^{k-1}(S)$. 

-- Suppose first $k < n$. We let $T$ be a $k$-face of $S$ and write, using $\rmd u = 0$ and Stokes:
\begin{align}
0 & = \int_{[W_S, T]} \rmd u,\\
&  = \sum_{U \in \subcells^{k-1}(T)} \orient([W_S, T], [W_S, U]) \int_{[W_S, U]} u,\\
& = \sum_{U \in \subcells^{k-1}(T)} \orient(T, U) \int_{[W_S, U]} u,
\end{align}
because the $k$-faces of $[W_S, T]$ are those containing $W_S$, in addition to $T$, where the integral of $u$ is $0$ by hypothesis.

This identity can be rewritten, in terms of the simplicial coboundary operator:
\begin{align}
\delta c_\bs = 0 \in \calC^{k}(S).
\end{align}

-- For $k=n$ this identity also holds, and just expresses that $\int_S u = 0$.

\tpoint For each vertex $V$ of $S$, let $\alpha_V$ denote the barycentric coordinate of $W_S$ in $S$.

Let $T$ be a $(k-2)$-face of $S$ and denote its vertices $V_0, \ldots, V_{k-2}$. We write, using that $u$ is $0$ at vertices of $S$, and summing over vertices $V$ in $S$ not in $T$:
\begin{align}
& \sum_{V \not \in T} \alpha_V \int_{[W_S, T, V]} u \\
= & \frac{1}{n!} \sum_{V \not \in T} \alpha_V u[W_S](V_0 - W_S, \ldots, V_{k-2} - W_S, V -W_S),\\
= &  \frac{1}{n!} u[W_S](V_0 - W_S, \ldots, V_{k-2} - W_S, \sum_{V \not \in T} \alpha_V (V -W_S)). 
\end{align}
Then we may substitute:
\begin{equation}
\sum_{V \not \in T} \alpha_V (V- W_S) = - \sum_{V \in T} \alpha_V  (V -W_S),
\end{equation}
which gives:
\begin{equation}
\sum_{V \not \in T} \alpha_V \int_{[W_S, T, V]} u = 0.
\end{equation}
This identity can be written:
\begin{equation}\label{eq:dprimealpha}
\sum_{U \in \subcells^{k-1}(S)} \alpha_{U \setminus T} \orient (U,T) c_{U} = 0.
\end{equation}

\tpoint
If it weren't for the weights $\alpha_{U\setminus T}$, this identity would be $\delta' c_\bs = 0$, where $\delta' : \calC^{k-1}(S) \to \calC^{k-2}(S)$ is the boundary operator, whose matrix in the canonical basis is the transpose of the matrix of $\delta$. Since $\delta c_\bs = 0$ and $\calC^\bs(S)$ is exact (at index $k-1 \geq 1$), we would conclude immediately that $c_\bs =0$. 

To account for the weights defined by $\alpha$, we define, on any subsimplex $T$ of $S$:
\begin{equation}
\alpha_T = \prod_{V \in \subcells^0(T)} \alpha_V,
\end{equation}
and rewrite (\ref{eq:dprimealpha}) as:
\begin{equation}
\sum_{U \in \subcells^{k-1}(S) } (\alpha_{U}/\alpha_T) \orient (U,T) c_{U} = 0.
\end{equation}
Let $\alpha_l$ be the operator  $\calC^l(S) \to \calC^l(S)$, whose matrix in the canonical basis is diagonal, with entry $\alpha_T$ at index $(T,T)$, $T \in \subcells^l(S)$. We obtain:
\begin{equation}
(\alpha_{k-2})^{-1} \delta' \alpha_{k-1} c_\bs = 0, 
\end{equation}
hence:
\begin{equation}
\delta' \alpha_{k-1} c_\bs = 0. 
\end{equation}

\tpoint Now, since $\delta c_\bs = 0$, we can choose $d_\bs \in \calC^{k-2}(S)$ such that $\delta d_\bs = c_\bs$. We have $\delta' \alpha_{k-1} \delta d_\bs = 0$. Since $\alpha_{k-1}$ is positive definite, we conclude $c_\bs = 0$.

\tpoint Since $u$ is $0$ at vertices, for any $(k-1)$-simplex $U = [V_1, \ldots V_k]$ we have:
\begin{equation}
 0 = c_U = (k+1)! u[W_S](V_1 - W_S , \ldots , V_k - W_S).
 \end{equation}
There are sufficiently many such $(k-1)$-simplexes to conclude that $u[W_S] = 0$.  
\end{proof}

The above result can also be applied to boundary simplexes. However in that case it will not give information about transverse components on the boundary (only the pullback to the boundary). Our second result will fill this gap. That is why the refinement used here is $\refi_k(S)$ not $\refi_{k-1}(S)$.

\begin{proposition}\label{prop:duconst}
Suppose $u\in \rmC^0 \poly^1 \difo^k(\refi_{k} (S))$ and that $\rmd u$ is constant on $S$. If $u$ is $0$ at the vertices of $S$ then $u$ is $0$ everywhere.
\end{proposition}
\begin{proof}
\tpoint For $k= 0$ the claim is just that an affine function is determined by its vertex values. So we suppose $ k \geq 1$ from now on.

\tpoint We proceed by induction on $\dim S$. Choose  $n\geq 1$ and suppose that the proposition has been proved for simplices $S$ with $\dim S < n$. We call this the outer induction hypothesis. Let $S$ be a simplex of dimension $\dim S = n$.

\tpoint For any $l$-face $T$ of $S$ with $l \leq k$, the trace of $u$ on $T$ is in $\rmC^0 \poly^1 (\refi_{k} (T)) \otimes \alter^k(\bbV)$ and since $T$ is not refined, $u$ is affine on $T$. Therefore the trace of $u$ is $0$. 

In particular the pullback of $u$ to any $k$-face is $0$. Therefore, for any $(k+1)$-face $T$ of $S$ the pullback $v$ of $u$ satisfies $\int_T \rmd v = 0$. The constant $\rmd u$ on $S$ has integral $0$ on all $(k+1)$-faces of $S$. Therefore $\rmd u = 0$ on $S$.

\tpoint Suppose we have proved that the pullback of $u$ to $l$-faces of $S$ is $0$, for some $l$ with $n > l \geq k$. We call this the inner induction hypothesis. Let $T$ be an $(l+1)$-face of $S$, and let $v$ be the pullback of $u$ to $T$. From Proposition \ref{prop:inpointzero} we conclude that $v(W_T)=0$. 

\noindent If $l = k$ then we conclude that $v=0$. \\
For $l >k$ we need to check that $v(W_{T'}) =0$ for faces $T'$ of $T$ of dimension $m$ with $k < m \leq l$.

-- Case $m=l$: Let $T'$ be an $l$-face of $T$. Let $w$ be the pullback to $T'$ of the $(k-1)$-form $v\ctr(W_T - W_{T'})$. We have $w \in \rmC^0 \poly^1 \difo^{k-1}(\refi_{k} (T'))$. We also notice that $\rmd w$ is piecewise constant on $T'$. Cartan's formula shows that $\rmd w$ is the derivative of $v$ in direction $(W_T - W_{T'})$, which is continuous. Therefore $\rmd w$ is constant. By the outer induction hypothesis, $w= 0$. Now we are in the situation that both $v$ and $v\ctr(W_T - W_{T'})$ have pullback $0$ to $T'$. Therefore $v(W_{T'}) = 0$.

-- Case $k < m <l$: If $T''$ is an $m$-face of $T$, for some $k < m <l$, $T''$ is included in at least two distinct $l$-faces of $T$. Since the pullback to these of $v$ is $0$, we deduce $v(W_{T''})=0$.

We deduce that $v=0$ on $T$. This completes the inner induction (on $l$), which may be followed up to the case $l = n-1$. There the conclusion is $u=0$, and this completes the outer induction step (on $n$).
\end{proof}

\begin{remark}\label{rem:scholie}
In other words the proposition says that if $u\in \rmC^0 \poly^1 \difo^k(\refi_{k} (S))$ and $\rmd u$ is constant, then $u$ is affine on $S$. The reciprocal is trivial.
\end{remark}

Finally we combine the preceding two propositions to prove the following.

\begin{proposition}\label{prop:unistep}
Suppose $u \in \rmC^0 \poly^1 \difo^k(\refi_{k-1} (S))$, and that $ \rmd u= 0$. If $u$ is zero at the vertices of $S$ and for any $k$-face $T$ of $S$, $\int_T u =0$, then $u = 0$.
\end{proposition}

\begin{proof}
We proceed by induction. We suppose that the proposition has been proved for any $S$ of dimension $n-1$, and let $S$ be a simplex of dimension $n$.

From Proposition \ref{prop:inpointzero}, we deduce that $u(W_S)=0$. 

For $k=n$ this is enough to conclude that $u=0$.

Suppose $k < n$. We want to check that $u(W_T)=0$ for any $m$-face $T$ of $S$ with $k \leq m <n$. We distinguish two cases for $m = \dim T$:

-- Case $m = n-1$: We know that the pullback of $u$ to $T$ is $0$ by the induction hypothesis.  Let $w$ be the pullback of $u \ctr(W_S - W_T)$ to $T$. Then $w$ is in $\rmC^0 \poly^1 \difo^{k-1}(\refi_{k-1} (T))$ and $\rmd w$ is constant.  From Proposition \ref{prop:duconst} it follows that $w$ is zero. We conclude that $u(W_T)=0$.

-- Case $k\leq  m< n-1$: Then $T$ is included in two distinct $(n-1)$-faces of $S$, on which the pullback of $u$ is zero. We deduce that $u(W_T)=0$.
\end{proof}

\section{Finite element spaces in high dimension\label{sec:spacehigh}}

\paragraph{A continuous finite element complex.} 
We consider the following spaces, on a simplex $S$, for $k \geq 1$:
\begin{equation}\label{def:kspace}
K^k(S)  = \{u \in \rmC^0 \poly^1 \difo^k(\refi_{k-1} (S)) \ : \ \rmd u = 0\}.
\end{equation}
For $k= 0$ we put:
\begin{align}
K^0(S) & = \{ u \in \rmC^0 \poly^1 \difo^0(\refi_{0} (S)) \ : \ \rmd u = 0\},\\
 & = \{ u : S \to \bbR \ : \ u \textrm{ is constant} \}.
\end{align}

We let $\poincare_S$ denote the Poincar\'e operator associated with the inpoint of $S$. 
We define the space of $k$-forms:
\begin{equation}\label{eq:adefk}
A^k(S) = K^k(S) + \poincare_{S} K^{k+1}(S).
\end{equation}
We want to prove that this choice provides a good finite element complex, in the sense that it defines a compatible finite element system.

We first notice:
\begin{proposition} We have that:\\
-- The sum (\ref{eq:adefk}) is direct.\\
-- The following sequence is exact:
\begin{equation}
\xymatrix{
0 \ar[r] & \bbR \ar[r] & A^0(S) \ar[r] & A^1(S) \ar[r] & \ldots  \ar[r] & A^n(S) \ar[r] & 0.
}
\end{equation}
\end{proposition}

\begin{proof}
Using essentially that the elements of $K^k(S)$ and $K^{k+1}(S)$ have $0$ exterior derivative.
\end{proof}

\begin{proposition}\label{prop:overdeta}
On $A^k(S)$ the degrees of freedom consisting of: \\
-- values at vertices,\\
-- values of the exterior derivative at vertices,\\
-- integrals on $k$-dimensional faces of $S$ (for $k \geq 1$), \\
overdetermine an element.
\end{proposition}
\begin{proof} Suppose $u$ is an element and that all these degrees of freedom are $0$. Applying Proposition \ref{prop:unistep} first to $\rmd u$ and then to $u$, one first gets that $\rmd u = 0$ and then that $u=0$.
\end{proof}
For $n = \dim S$ and $k \geq 1$, this gives the upper bounds:
\begin{align}
\dim A^k(S) &\leq (n+1) ({n \choose k} + {n \choose k+1} ) + { n+1 \choose k + 1} =  (n+2){ n+1 \choose k + 1},
\end{align}
and:
\begin{equation}
\dim A^0(S) \leq  (n+1)^2.
\end{equation}
To get unisolvence of the degrees of freedom, we would like to prove the converse bounds on dimension.

We do this for the case of a generalized Powell-Sabin split in dimension $n = 3$. We want to prove:
\begin{align}
\dim A^3(S) & = 5,\\
\dim A^2(S) & = 20,\\
\dim A^1(S) & = 30, \\
\dim A^0(S) & = 16.
\end{align}
This amounts to:
\begin{align}
\dim K^3(S) & = 5,\\
\dim K^2(S) & = 15,\\
\dim K^1(S) & = 15, \\
\dim K^0(S) & = 1.
\end{align}

\begin{proposition}\label{prop:dimcount}
The above dimension counts, in dimension $n= 3$, are correct.
\end{proposition}
\begin{proof}
\tpoint For  $K^3(S)$ and $K^0(S)$ it is clear.

\tpoint For $K^2(S)$  we can get the lower bound as follows: The tetrahedron and its faces are each equipped with an inpoint. So $\rmC^0\poly^1 \difo^2(\refi_1(S))$ has dimension $(1+4+4)\cdot 3 = 27$. On the other hand there are 12 subtetrahedra on which we enforce one condition. So $\dim K^2(S) \geq 27 -12 = 15$.

\tpoint For $K^1(S)$ all faces and edges are refined, so $\rmC^0\poly^1 \difo^1(\refi_0(S))$ has dimension 45. To enforce on an element $u$ of $\rmC^0\poly^1 \difo^1(\refi_0(S))$, that  $\rmd u = 0$, we use that $\rmd u $ is constant on each of the 24 small tetrahedra of $\refi_0(S)$. Therefore it is enough to enforce the pullback to be zero on a set of triangular faces in $\refi_0(S)$, such that each litte tetrahedron has three of them in its boundary. We choose these triangles as follows:

-- We impose that the pullback of $\rmd u $ to the triangles joining the inpoint of the tetrahedron, the inpoint of a face and a vertex should be zero. These are 3 conditions per face, and there are 4 faces.

-- For each edge, the inpoints of the tetrahedron, the two adjacent triangular faces, and the edge itself are coplanar, by the choice of split. So we may use Proposition \ref{prop:alired} to impose only 3 conditions, rather than 4. There are 6 edges.

This gives $\dim K^1(S) \geq 45 - 4 \cdot 3 - 6 \cdot 3 = 15$.
\end{proof}

In arbitrary dimension $n$ we can still be precise about the last two spaces in the complex, which are those relevant for Stokes. 

\begin{proposition}\label{prop:nonegood} The given degrees of freedom on $A^{n-1}(S)$ and $A^n(S)$ are unisolvent. The dimensions are $\dim A^{n-1}(S) = (n+1)(n+2)$ and $\dim A^{n}(S) = n+2$.

The associated interpolator commutes with the divergence operator.

This gives a minimal good element for continous vectorfields with continuous divergence. 
\end{proposition}
\begin{proof}
\tpoint We have $\dim A^n(S) = n+2$, since there are $n+1$ vertices in $S$ and we have added the inpoint of $S$.

\tpoint For $K^{n-1}(S)$ we may estimate its dimension as follows. In the refinement $\refi_{n-2} (S)$ there are the $n+1$ vertices of $S$, the $n+1$ inpoints attached to $(n-1)$-faces, and one inpoint in $S$. This gives:
\begin{equation}
\dim  \rmC^0 \poly^1 \difo^{n-1}(\refi_{n-2} (S)) = n( 2(n+1) +1), 
\end{equation}
There are also $(n+1)n$ small $n$-simplexes, on which we express $\rmd u= 0$ as one scalar constraint. This gives:
\begin{equation}
\dim K^{n-1}(S) \geq n( 2(n+1) +1) - (n+1)n = n^2 + 2n.
\end{equation}

\tpoint We conclude:
\begin{equation}
\dim A^{n-1}(S) \geq n^2 + 2n + n+2 = (n+1)(n+2).
 \end{equation}

\tpoint Since these lower bounds coincide with the number of degrees of freedom, and these are overdetermining, the degrees of freedom are unisolvent, and the dimension count follows.
\end{proof}
For general $n$, the analysis of the complex at lower indices seems more complicated, say for the space $A^1(S)$.

\paragraph{Behavior on faces.}
We are now interested in determining the restrictions to the faces of $S$, of the spaces $A^k(S)$. This is important for the inter-element continuity of fields, to get global fields of the required regularity. It is also inherent to the framework of finite element system, which encodes the inter-element continuity by taking an inverse limit. 

For this purpose, some alternative characterisations of $A^k(S)$ are sometimes useful. We let $\koszul_S$ denote the Koszul operator associated with the inpoint of $S$. We have:
\begin{equation}
\poincare_{S} K^{k+1}(S) = \koszul_{S} K^{k+1}(S).
\end{equation}
It follows that:
\begin{equation}
A^k(S) = K^k(S) + \koszul_{S} K^{k+1}(S).
\end{equation}
We also have the alternative characterization:
\begin{proposition}
We have:
\begin{align}\label{eq:altdefa}
A^k(S) = \{ &  u \in \rmC^0\poly^2\difo^k(\refi_{k-1} (S)) \ : \ \rmd u \in \rmC^0 \poly^1 \difo^{k+1}(\refi_k (S)) \textrm{ and }\\
&  u - \koszul_S \rmd u \in  \rmC^0\poly^1\difo^{k}(\refi_{k-1} (S)) \},
\end{align}
\end{proposition}
\begin{proof}
Using (\ref{eq:poincare1}).
\end{proof}

\begin{proposition}\label{prop:kcaffine}
 For any element $u$ of $K^k(S)$, if $T$ is face of $S$, then for any face $U$ of $S$ with $T \subcell U \subcell S$, the pullback of $u\ctr (W_S - W_U)$ to $T$ is affine.
\end{proposition}
\begin{proof}
It suffices to show that the pullback of $u\ctr (W_S - W_U)$ to $U$ is affine. Let $v$ be the pullback of $u\ctr (W_S - W_U)$ to the simplex $[W_S, U]$. We have that $\rmd v$ is piecewise constant and continuous, as the Lie derivative of $u$ along $W_S - W_U$. Hence $\rmd v$ is contant. We have $\pull_U v \in \rmC^0 \poly^1 \difo^k(\refi_{k} (U)) $  and may apply Remark \ref{rem:scholie}.
\end{proof}

\begin{proposition} For any element $u$ of $A^{k}(S)$, if $T$ is face of $S$, then for any face $U$ of $S$ with $T \subcell U \subcell S$, the pullback of $(u - \koszul_T \rmd u)\ctr (W_S - W_U)$ to $T$ is affine.
\end{proposition}
\begin{proof}
\tpoint We put $v = \rmd u$. We first examine the pullback of $(u - \koszul_U v)\ctr (W_S - W_U)$ to $U$. We have:
\begin{align}
(u - \koszul_U v)\ctr (W_S - W_U) & = (u - \koszul_S v  + \koszul_S v - \koszul_U v)\ctr (W_S - W_U), \nonumber\\
& = (u - \koszul_S v)\ctr (W_S - W_U) - \nonumber\\
& \phantom{=} \quad  (v \ctr (W_S - W_U))\ctr (W_S - W_U).
\end{align}
The term on the last line is $0$. Since $u - \koszul_S v \in K^k(S)$ we may apply the preceding proposition to it. We deduce that the pullback of $(u - \koszul_U v)\ctr (W_S - W_U)$ to $U$ is affine.

\tpoint Now on $T$ we write:
\begin{equation}
u - \koszul_T v = u- \koszul_U v + v \ctr(W_U - W_T).
\end{equation}
In the right hand side, we remark that $(u- \koszul_U v) \ctr (W_S - W_U)$ has a pullback to $T$ which is affine, by the preceding point. Then we consider $w = v \ctr(W_U - W_T) \ctr (W_S - W_U)$. From the preceding proposition $v \ctr (W_S - W_U)$ is affine when pulled back on $U$. Hence $w$ pulled back to $T$ is also affine.
\end{proof}

On lower dimensional subcells $T$ of $S$ we can define first:
\begin{equation}
M^k(T) = K^k(T) + \koszul_{T} K^{k+1}(T).
\end{equation}
We have:
\begin{proposition}
For any simplexes $T \subcell U$ in $\subcells(S)$, the pullback operator gives a map $\pull_T : M^k(U) \to M^k(T)$.
\end{proposition}
\begin{proof}
We use the characterization (\ref{eq:altdefa}) which applies also to $M^k(T)$.

Choose $u \in M^k(U)$ and put $v = \pull_T u$. We have:
\begin{align}
v - \koszul_T \rmd v & = v - \pull_T (\rmd u \ctr X_T),\\
& = \pull_T ( u - \koszul_U \rmd u) - \pull_T ( \rmd u \ctr (W_U - W_T)).
\end{align}
The first term in this difference is in $M^k(T)$ by the characterization (\ref{eq:altdefa}) applied to $M^k(U)$ and $M^k(T)$. The second term is affine on $T$, by applying Proposition \ref{prop:kcaffine} to $\rmd u \in K^{k+1}(U)$, so it's also in $M^k(T)$.
\end{proof}

Hence $M$ defines a finite element system with respect to pull-backs. However this is not the restriction operator that interests us for the Stokes equation.

It seems useful to define:
\begin{equation}
\bbW_T = \myspan \{ W_S - W_U \ : \  U \in \subcells(S) \textrm{  and } T \subcell U \subcell S \}.
\end{equation}
Motivated by the above considerations we define, for any simplex $T \in \subcells(S)$:
\begin{align}
A^k(T) = \{ & (u,v) \in \rmC^0\poly^2(\refi_{k-1} (T)) \otimes \alter^k(\bbV) \, \oplus \, \rmC^0\poly^1(\refi_{k} (T)) \otimes \alter^{k+1}(\bbV) \ : \ \nonumber\\ 
& (u,v) \textrm{ is admissible and} \pull_T u \in M^k(T) \textrm{ and}\nonumber\\
& \forall Y \in \bbW_T \quad \pull_T(v \ctr Y) \textrm{ and } \pull_T((u - \kappa_T v) \ctr Y) \textrm{ are affine.} \} .
\end{align}
When $T$ is a vertex $V$ this definition reduces to:
\begin{equation}
A^k(V) = \alter^k(\bbV) \oplus \alter^{k+1}(\bbV).
\end{equation}

\begin{proposition}
The spaces $A^k(T)$ constitute a finite element system, with respect to restrictions which are double-traces and differential (\ref{eq:difadm}).
\end{proposition}
\begin{proof}
That restrictions map from $A^k(S)$ to $A^k(T)$ was proved in the preceding three propositions.

That they also map from $A^k(U)$ to $A^k(T)$ when $T \subcell U \subcell S$ follows from similar arguments. 

Stability under the differential is straightforward.
\end{proof}

\begin{proposition} \label{prop:azero}
Suppose $T \in \subcells(S)$ is not a vertex. If $k = \dim T$ we have:
\begin{equation}
\dim A^k_0(T) \leq 1.
\end{equation}
If $k \neq \dim T$, $A^k_0(T) = 0$.
\end{proposition}
\begin{proof}
Suppose $(u,v) \in A^k_0(T)$ and that, in case $k = \dim T$, we have $\int_T \pull_T u = 0$.

By Proposition \ref{prop:overdeta} we get $\pull_T u = 0$ and $\pull_T v = 0$.

Then we get that, whenever $Y \in \bbW_T$,  $\pull_T(v \ctr Y) = 0$ and $\pull_T(u \ctr Y) = 0$, since they are affine and have trace $0$ on $\partial T$. 

Since $\bbW_T + \vect T  = \bbV$, the two conditions above give $u= 0$ and $v=0$.
\end{proof}

\begin{theorem}\label{theo:3dstokes}
The finite element system $A$ is compatible, when $n = 3$, and the split is Powell-Sabin/Worsey-Piper.
\end{theorem}
\begin{proof}
From Propositions \ref{prop:dimcount} and \ref{prop:azero} we get by computing: 
\begin{equation}
\dim A^k(S) \geq \sum_{T \in \subcells(S)} \dim A^k_0(T).
\end{equation}
Then Proposition \ref{prop:extdim} shows that equality holds and that the finite element system is flabby. In particular $\dim A^k_0(T) = 1$ for $k = \dim T$, and the integral provides an isomorphism to $\bbR$.

The cohomology of the sequence $A^\bs_0(T)$ is then trivially determined. One concludes by Theorem \ref{theo:altcomp}.
\end{proof}

{
\begin{center}
\begin{figure}
\setlength{\unitlength}{1.2cm}
\begin{picture}(2,2)(-0.5,0)
\put(0,0){
\begin{picture}(2,2)

{\color{red}\polygon*(-0.1, 0.6)(0, 0.9)(0, 1.5)
%{\color{lightorange}\polygon*(-0.1, 0.6)(0, 1.5)(0.368, 0.748)}
\color{lightorange}\polygon*(-0.1, 0.6)(0.368, 0.748)(0, 0.9)
\color{yellow}\polygon*(0, 0.9)(0, 1.5)(0.368, 0.748)
}

\put(-0.5, 0.5){\circle*{0.1}}
\put(-0.5, 0.5){\circle{0.2}}
\put(1, 0.5){\circle*{0.1}}
\put(1, 0.5){\circle{0.2}}
\put(0, 1.5){\circle*{0.1}}
\put(0, 1.5){\circle{0.2}}
\put(0, 0){\circle*{0.1}}
\put(0, 0){\circle{0.2}}
\put(0, 0){\line(-1, 1){0.5}}
\put(0, 0){\line(2, 1){1}}
\put(1, 0.5){\line(-1, 1){1}}
\put(-0.5, 0.5){\line(1, 2){0.5}}
\put(0, 0){\line(0, 1){1.5}}
\put(-0.1, 0.6){\circle*{0.05}}

\multiput(-0.5,0.5)(0.05,0){30}{\circle*{0.03}}

\put(0, 0){\line(1, 2){0.5}}
\put(0, 1.5){\line(1, -2){0.6}}
\put(1, 0.5){\line(-5, 2){1}}
\multiput(-0.1, 0.6)(0.05,0.015){10}{\circle*{0.02}}
\multiput(-0.1, 0.6)(0.05,-0.021){15}{\circle*{0.02}}
\multiput(-0.1, 0.6)(0.022,0.06){6}{\circle*{0.02}}
\multiput(-0.1, 0.6)(0.005, 0.05){18}{\circle*{0.02}}
\multiput(-0.1, 0.6)(0.008, -0.05){12}{\circle*{0.02}}
\multiput(-0.1, 0.6)(0.05, -0.004){20}{\circle*{0.02}}
\multiput(-0.1, 0.6)(0.048, 0.033){13}{\circle*{0.02}}
\end{picture}
}

\put(1.3, 0.8){\vector(1, 0){0.8}}
\put(1.44, 0.92){\footnotesize{$\grad$}}

\put(3,0){
\begin{picture}(2,2)

{\color{red}\polygon*(-0.1, 0.6)(0, 0.9)(0, 1.5)
%{\color{lightorange}\polygon*(-0.1, 0.6)(0, 1.5)(0.368, 0.748)}
\color{lightorange}\polygon*(-0.1, 0.6)(0.368, 0.748)(0, 0.9)
\color{yellow}\polygon*(0, 0.9)(0, 1.5)(0.368, 0.748)
}

\put(-0.5, 0.5){\circle*{0.1}}
%\put(-0.5, 0.5){\circle{0.2}}
\put(1, 0.5){\circle*{0.1}}
%\put(1, 0.5){\circle{0.2}}
\put(0, 1.5){\circle*{0.1}}
%\put(0, 1.5){\circle{0.2}}
\put(0, 0){\circle*{0.1}}
%\put(0, 0){\circle{0.2}}
\put(0, 0){\line(-1, 1){0.5}}
\put(0, 0){\line(2, 1){1}}
\put(1, 0.5){\line(-1, 1){1}}
\put(-0.5, 0.5){\line(1, 2){0.5}}
\put(0, 0){\line(0, 1){1.5}}
\put(-0.1, 0.6){\circle*{0.05}}

{
\linethickness{0.3mm}
\put(0,0.4){\vector(0, 1){0.5}}
\put(0.25, 0.125){\vector(2, 1){0.5}}
\put(0.8, 0.7){\vector(-1, 1){0.5}}
}

\put(-0.2, 1.65){\footnotesize{$\curl$}}
\put(-0.9, 0.25){\footnotesize{$\curl$}}
\put(-0.2, -0.25){\footnotesize{$\curl$}}
\put(0.8, 0.25){\footnotesize{$\curl$}}

\multiput(-0.5,0.5)(0.05,0){30}{\circle*{0.03}}

\put(0, 0){\line(1, 2){0.5}}
\put(0, 1.5){\line(1, -2){0.6}}
\put(1, 0.5){\line(-5, 2){1}}
\multiput(-0.1, 0.6)(0.05,0.015){10}{\circle*{0.02}}
\multiput(-0.1, 0.6)(0.05,-0.021){15}{\circle*{0.02}}
\multiput(-0.1, 0.6)(0.022,0.06){6}{\circle*{0.02}}
\multiput(-0.1, 0.6)(0.005, 0.05){18}{\circle*{0.02}}
\multiput(-0.1, 0.6)(0.008, -0.05){12}{\circle*{0.02}}
\multiput(-0.1, 0.6)(0.05, -0.004){20}{\circle*{0.02}}
\multiput(-0.1, 0.6)(0.048, 0.033){13}{\circle*{0.02}}
\end{picture}
}

\put(4.3, 0.8){\vector(1, 0){0.8}}
\put(4.44, 0.92){\footnotesize{$\curl$}}

\put(6,0){
\begin{picture}(2,2)

%{\color{orange}\polygon*(-0.1, 0.6)(0, 1.5)(0.368, 0.748)(0, 0)}

{
%{\color{lightorange}\polygon*(-0.1, 0.6)(0, 1.5)(0.368, 0.748)}
\color{yellow}\polygon*(0, 0)(0, 1.5)(0.368, 0.748)
\color{red}\polygon*(0, 0)(0, 1.5)(-0.1, 0.6)
}

\put(0.368, 0.748){\vector(5, 2){0.3}}

\put(-0.5, 0.5){\circle*{0.1}}
%\put(-0.5, 0.5){\circle{0.2}}
\put(1, 0.5){\circle*{0.1}}
%\put(1, 0.5){\circle{0.2}}
\put(0, 1.5){\circle*{0.1}}
%\put(0, 1.5){\circle{0.2}}
\put(0, 0){\circle*{0.1}}
%\put(0, 0){\circle{0.2}}
\put(0, 0){\line(-1, 1){0.5}}
\put(0, 0){\line(2, 1){1}}

\put(-0.2, 1.65){\footnotesize{$\div$}}
\put(-0.9, 0.25){\footnotesize{$\div$}}
\put(-0.2, -0.25){\footnotesize{$\div$}}
\put(0.8, 0.25){\footnotesize{$\div$}}

\put(1, 0.5){\line(-1, 1){1}}
\put(-0.5, 0.5){\line(1, 2){0.5}}
\put(0, 0){\line(0, 1){1.5}}
\put(-0.1, 0.6){\circle*{0.05}}

\multiput(-0.5,0.5)(0.05,0){30}{\circle*{0.03}}

\put(0, 0){\line(1, 2){0.38}}
\put(0, 1.5){\line(1, -2){0.38}}
\put(1, 0.5){\line(-5, 2){0.62}}
\multiput(-0.1, 0.6)(0.05,0.015){10}{\circle*{0.02}}
%\multiput(-0.1, 0.6)(0.05,-0.021){15}{\circle*{0.02}}
%\multiput(-0.1, 0.6)(0.022,0.06){6}{\circle*{0.02}}
\multiput(-0.1, 0.6)(0.005, 0.05){18}{\circle*{0.02}}
\multiput(-0.1, 0.6)(0.008, -0.05){12}{\circle*{0.02}}
\multiput(-0.1, 0.6)(0.05, -0.004){20}{\circle*{0.02}}
%\multiput(-0.1, 0.6)(0.048, 0.033){13}{\circle*{0.02}}
\end{picture}

}

\put(7.3, 0.8){\vector(1, 0){0.8}}
\put(7.44, 0.92){\footnotesize{$\div$}}

\put(9,0){
\begin{picture}(2,2)

{
%{\color{lightorange}\polygon*(-0.1, 0.6)(0, 1.5)(0.368, 0.748)}
\color{yellow}\polygon*(0, 0)(0, 1.5)(1, 0.5)
\color{red}\polygon*(0, 0)(0, 1.5)(-0.1, 0.6)
}

\put(-0.5, 0.5){\circle*{0.1}}
%\put(-0.5, 0.5){\circle{0.2}}
\put(1, 0.5){\circle*{0.1}}
%\put(1, 0.5){\circle{0.2}}
\put(0, 1.5){\circle*{0.1}}
%\put(0, 1.5){\circle{0.2}}
\put(0, 0){\circle*{0.1}}
%\put(0, 0){\circle{0.2}}
\put(0, 0){\line(-1, 1){0.5}}
\put(0, 0){\line(2, 1){1}}
\put(1, 0.5){\line(-1, 1){1}}
\put(-0.5, 0.5){\line(1, 2){0.5}}
\put(0, 0){\line(0, 1){1.5}}
\put(-0.1, 0.6){\circle*{0.05}}

\multiput(-0.5,0.5)(0.05,0){30}{\circle*{0.03}}

%\multiput(-0.1, 0.6)(0.05,0.015){10}{\circle*{0.02}}
%\multiput(-0.1, 0.6)(0.05,-0.021){15}{\circle*{0.02}}
%\multiput(-0.1, 0.6)(0.022,0.06){6}{\circle*{0.02}}
\multiput(-0.1, 0.6)(0.005, 0.05){18}{\circle*{0.02}}
\multiput(-0.1, 0.6)(0.008, -0.05){12}{\circle*{0.02}}
\multiput(-0.1, 0.6)(0.05, -0.004){20}{\circle*{0.02}}
%\multiput(-0.1, 0.6)(0.048, 0.033){13}{\circle*{0.02}}
\end{picture}

}

\end{picture}
\caption{Finite element complex described in Theorem \ref{theo:3dstokes}.\newline
Gives a Stokes pair with continuous pressure.  }
\end{figure}
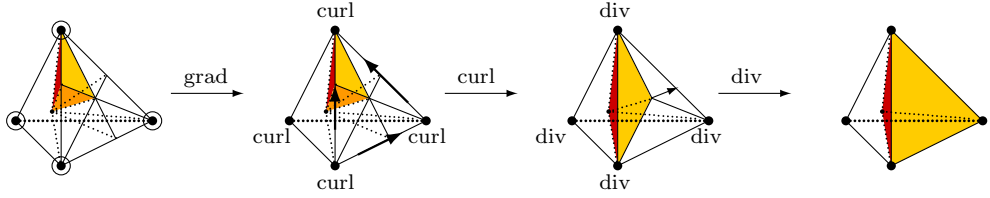
\end{center}
}

\begin{remark}
One could also check:
\begin{equation}
\{\pull_T u \ : \ (u,v) \in A^\bs(T) \} = M^k(T),
\end{equation}
and deduce from there that the sequences $A^\bs(T)$ resolve $\bbR$ by Lemma \ref{lem:exind}.
\end{remark}

\begin{remark}
A crucial question, to address the case of general $n$, is whether one has $\bbW_T\cap \vect T = 0$. It seems that the condition that the following sums are direct:
\begin{equation}
\bbW_T \oplus \vect T  = \bbV,
\end{equation}
captures the sort of alignment conditions one needs to impose. Even though $\bbW_T$ was defined in terms of the choices of inpoints in $S$, when several $n$-dimensional simplices meet at $T$, they should determine the same $\bbW_T$.
\end{remark}

\paragraph{Branching into Whithey forms.}

Consider the case of arbitrary $\dim S = n$. Let $\difo^k(S)$ denote the space of constant $k$-forms on $S$. Fix an index $\ell \in [0, n]$. Instead of (\ref{eq:adefk}) we define:
\begin{equation}\label{eq:adefkbis}
A^k(S) = \left\{ \begin{array}{ll}
K^k(S) + \poincare_{S} K^{k+1}(S),& k < \ell,\\
K^k(S) + \poincare_{S} \difo^{k+1}(S),& k = \ell,\\
\difo^{k}(S) +  \poincare_{S} \difo^{k+1}(S) &  k > \ell.
\end{array} \right.
\end{equation}
In this definition we recognize $\difo^{k}(S) +  \poincare_{S} \difo^{k+1}(S)$ as the space of Whitney $k$-forms on $S$, henceforth denoted $\whitney^k(S)$. Its canonical choice of degrees of freedom consists of integrals on $k$-dimensional faces.

\begin{proposition} We have that:\\
-- The sums in (\ref{eq:adefkbis}) are direct.\\
-- The following sequence is exact:
\begin{equation}
\xymatrix{
0 \ar[r] & \bbR \ar[r] & A^0(S) \ar[r] & A^1(S) \ar[r] & \ldots  \ar[r] & A^n(S) \ar[r] & 0.
}
\end{equation}
\end{proposition}

The only new space in the above sequence is the one attached to the index $\ell$: before $\ell$ we have the space studied in the previous paragraph and after $\ell$ we have Whitney forms.
\begin{proposition}\label{prop:dofell}
On $A^\ell(S)$ the degrees of freedom consisting of:\\
-- evaluation at vertices,\\
-- integrals of pullback to $\ell$-dimensional faces of $T$,\\
overdetermine an element.
\end{proposition} 
\begin{proof}
If $u \in A^\ell(S)$ has all its degrees of freedom equal to $0$, then one checks first that $\rmd u = 0$, from the theory of Whitney forms. Then one deduces that $u = 0$ from Proposition \ref{prop:unistep}.
\end{proof}

To get a finite element system we define first, for $k = \ell$:
\begin{align}
N^k(T) = \{ &  u \in \rmC^0\poly^2\difo^k(\refi_{k-1} (T)) \ : \ \rmd u \in \difo^{k+1}(T) \textrm{ and }\\
&  u - \koszul_T \rmd u \in  \rmC^0\poly^1\difo^{k}(\refi_{k-1} (T)) \}.
\end{align}
For $T \in \subcells(S)$ we then put:
\begin{align}
A^k(T) = \{ & u \in \rmC^0\poly^2(\refi_{k-1} (T)) \otimes \alter^k(\bbV)  \ : \ \pull_T u \in N^k(T) \textrm{ and}\nonumber\\
& \forall Y \in \bbW_T \quad \pull_T(u \ctr Y)  \textrm{ is affine.} \} .
\end{align}
For $k < \ell$ one uses the previously defined spaces. For $k > \ell$ one uses Whitney forms. This gives a finite element system.

\begin{proposition} In the case $n= 3$ the degrees of freedom described in Proposition \ref{prop:dofell} are unisolvent and we get two new compatible finite element systems for $\ell = 2$ and $\ell = 1$.
\end{proposition}
\begin{proof} We use Proposition \ref{prop:dimcount}.
When we choose to branch at $\ell = 2$, we have: 
\begin{equation}
\dim A^\ell(S) =  15 + 1 = 16 = 4 \times 3 + 4. 
\end{equation}
When we choose to branch at $\ell = 1$ we have:
\begin{equation} 
\dim A^\ell(S) = 15 + 3 = 18 = 4 \times 3 + 6.
\end{equation}
In both cases this proves unisolvence.
\end{proof}

The case $\ell = 2$ of this proposition is described in Figure \ref{fig:3dstokesl2} and the case $\ell = 1$ is decribed in Figure \ref{fig:3dstokesl1}.

{
\begin{center}
\begin{figure}
\setlength{\unitlength}{1.2cm}
\begin{picture}(2,2)(-0.5,0)
\put(0,0){
\begin{picture}(2,2)

{\color{red}\polygon*(-0.1, 0.6)(0, 0.9)(0, 1.5)
%{\color{lightorange}\polygon*(-0.1, 0.6)(0, 1.5)(0.368, 0.748)}
\color{lightorange}\polygon*(-0.1, 0.6)(0.368, 0.748)(0, 0.9)
\color{yellow}\polygon*(0, 0.9)(0, 1.5)(0.368, 0.748)
}
\put(-0.5, 0.5){\circle*{0.1}}
\put(-0.5, 0.5){\circle{0.2}}
\put(1, 0.5){\circle*{0.1}}
\put(1, 0.5){\circle{0.2}}
\put(0, 1.5){\circle*{0.1}}
\put(0, 1.5){\circle{0.2}}
\put(0, 0){\circle*{0.1}}
\put(0, 0){\circle{0.2}}
\put(0, 0){\line(-1, 1){0.5}}
\put(0, 0){\line(2, 1){1}}
\put(1, 0.5){\line(-1, 1){1}}
\put(-0.5, 0.5){\line(1, 2){0.5}}
\put(0, 0){\line(0, 1){1.5}}
\put(-0.1, 0.6){\circle*{0.05}}

\multiput(-0.5,0.5)(0.05,0){30}{\circle*{0.03}}

\put(0, 0){\line(1, 2){0.5}}
\put(0, 1.5){\line(1, -2){0.6}}
\put(1, 0.5){\line(-5, 2){1}}
\multiput(-0.1, 0.6)(0.05,0.015){10}{\circle*{0.02}}
\multiput(-0.1, 0.6)(0.05,-0.021){15}{\circle*{0.02}}
\multiput(-0.1, 0.6)(0.022,0.06){6}{\circle*{0.02}}
\multiput(-0.1, 0.6)(0.005, 0.05){18}{\circle*{0.02}}
\multiput(-0.1, 0.6)(0.008, -0.05){12}{\circle*{0.02}}
\multiput(-0.1, 0.6)(0.05, -0.004){20}{\circle*{0.02}}
\multiput(-0.1, 0.6)(0.048, 0.033){13}{\circle*{0.02}}
\end{picture}
}

\put(1.3, 0.8){\vector(1, 0){0.8}}
\put(1.44, 0.92){\footnotesize{$\grad$}}

\put(3,0){
\begin{picture}(2,2)

{\color{red}\polygon*(-0.1, 0.6)(0, 0.9)(0, 1.5)
%{\color{lightorange}\polygon*(-0.1, 0.6)(0, 1.5)(0.368, 0.748)}
\color{lightorange}\polygon*(-0.1, 0.6)(0.368, 0.748)(0, 0.9)
\color{yellow}\polygon*(0, 0.9)(0, 1.5)(0.368, 0.748)
}
\put(-0.5, 0.5){\circle*{0.1}}
%\put(-0.5, 0.5){\circle{0.2}}
\put(1, 0.5){\circle*{0.1}}
%\put(1, 0.5){\circle{0.2}}
\put(0, 1.5){\circle*{0.1}}
%\put(0, 1.5){\circle{0.2}}
\put(0, 0){\circle*{0.1}}
%\put(0, 0){\circle{0.2}}
\put(0, 0){\line(-1, 1){0.5}}
\put(0, 0){\line(2, 1){1}}
\put(1, 0.5){\line(-1, 1){1}}
\put(-0.5, 0.5){\line(1, 2){0.5}}
\put(0, 0){\line(0, 1){1.5}}
\put(-0.1, 0.6){\circle*{0.05}}

{
\linethickness{0.3mm}
\put(0,0.4){\vector(0, 1){0.5}}
\put(0.25, 0.125){\vector(2, 1){0.5}}
\put(0.8, 0.7){\vector(-1, 1){0.5}}
}

\put(-0.2, 1.65){\footnotesize{$\curl$}}
\put(-0.9, 0.25){\footnotesize{$\curl$}}
\put(-0.2, -0.25){\footnotesize{$\curl$}}
\put(0.8, 0.25){\footnotesize{$\curl$}}

\multiput(-0.5,0.5)(0.05,0){30}{\circle*{0.03}}

\put(0, 0){\line(1, 2){0.5}}
\put(0, 1.5){\line(1, -2){0.6}}
\put(1, 0.5){\line(-5, 2){1}}
\multiput(-0.1, 0.6)(0.05,0.015){10}{\circle*{0.02}}
\multiput(-0.1, 0.6)(0.05,-0.021){15}{\circle*{0.02}}
\multiput(-0.1, 0.6)(0.022,0.06){6}{\circle*{0.02}}
\multiput(-0.1, 0.6)(0.005, 0.05){18}{\circle*{0.02}}
\multiput(-0.1, 0.6)(0.008, -0.05){12}{\circle*{0.02}}
\multiput(-0.1, 0.6)(0.05, -0.004){20}{\circle*{0.02}}
\multiput(-0.1, 0.6)(0.048, 0.033){13}{\circle*{0.02}}
\end{picture}
}

\put(4.3, 0.8){\vector(1, 0){0.8}}
\put(4.44, 0.92){\footnotesize{$\curl$}}

\put(6,0){
\begin{picture}(2,2)

{
%{\color{lightorange}\polygon*(-0.1, 0.6)(0, 1.5)(0.368, 0.748)}
\color{yellow}\polygon*(0, 0)(0, 1.5)(0.368, 0.748)
\color{red}\polygon*(0, 0)(0, 1.5)(-0.1, 0.6)
}
\put(0.368, 0.748){\vector(5, 2){0.3}}

\put(-0.5, 0.5){\circle*{0.1}}
%\put(-0.5, 0.5){\circle{0.2}}
\put(1, 0.5){\circle*{0.1}}
%\put(1, 0.5){\circle{0.2}}
\put(0, 1.5){\circle*{0.1}}
%\put(0, 1.5){\circle{0.2}}
\put(0, 0){\circle*{0.1}}
%\put(0, 0){\circle{0.2}}
\put(0, 0){\line(-1, 1){0.5}}
\put(0, 0){\line(2, 1){1}}

%\put(-0.2, 1.65){\footnotesize{$\div$}}
%\put(-0.9, 0.25){\footnotesize{$\div$}}
%\put(-0.2, -0.25){\footnotesize{$\div$}}
%\put(0.8, 0.25){\footnotesize{$\div$}}

\put(1, 0.5){\line(-1, 1){1}}
\put(-0.5, 0.5){\line(1, 2){0.5}}
\put(0, 0){\line(0, 1){1.5}}
\put(-0.1, 0.6){\circle*{0.05}}

\multiput(-0.5,0.5)(0.05,0){30}{\circle*{0.03}}

\put(0, 0){\line(1, 2){0.38}}
\put(0, 1.5){\line(1, -2){0.38}}
\put(1, 0.5){\line(-5, 2){0.62}}
\multiput(-0.1, 0.6)(0.05,0.015){10}{\circle*{0.02}}
%\multiput(-0.1, 0.6)(0.05,-0.021){15}{\circle*{0.02}}
%\multiput(-0.1, 0.6)(0.022,0.06){6}{\circle*{0.02}}
\multiput(-0.1, 0.6)(0.005, 0.05){18}{\circle*{0.02}}
\multiput(-0.1, 0.6)(0.008, -0.05){12}{\circle*{0.02}}
\multiput(-0.1, 0.6)(0.05, -0.004){20}{\circle*{0.02}}
%\multiput(-0.1, 0.6)(0.048, 0.033){13}{\circle*{0.02}}
\end{picture}

}

\put(7.3, 0.8){\vector(1, 0){0.8}}
\put(7.44, 0.92){\footnotesize{$\div$}}

\put(9,0){
\begin{picture}(2,2)

{\color{yellow}\polygon*(0, 1.5)(1, 0.5)(0, 0)
\color{darkorange}\polygon*(-0.5, 0.5)(0, 1.5)(0, 0)}

%\put(-0.5, 0.5){\circle*{0.1}}
%\put(-0.5, 0.5){\circle{0.2}}
%\put(1, 0.5){\circle*{0.1}}
%\put(1, 0.5){\circle{0.2}}
%\put(0, 1.5){\circle*{0.1}}
%\put(0, 1.5){\circle{0.2}}
%\put(0, 0){\circle*{0.1}}
%\put(0, 0){\circle{0.2}}
\put(0, 0){\line(-1, 1){0.5}}
\put(0, 0){\line(2, 1){1}}
\put(1, 0.5){\line(-1, 1){1}}
\put(-0.5, 0.5){\line(1, 2){0.5}}
\put(0, 0){\line(0, 1){1.5}}
%\put(-0.1, 0.6){\circle*{0.05}}

\multiput(-0.5,0.5)(0.05,0){30}{\circle*{0.03}}

%\multiput(-0.1, 0.6)(0.05,0.015){10}{\circle*{0.02}}
%\multiput(-0.1, 0.6)(0.05,-0.021){15}{\circle*{0.02}}
%\multiput(-0.1, 0.6)(0.022,0.06){6}{\circle*{0.02}}
%\multiput(-0.1, 0.6)(0.005, 0.05){18}{\circle*{0.02}}
%\multiput(-0.1, 0.6)(0.008, -0.05){12}{\circle*{0.02}}
%\multiput(-0.1, 0.6)(0.05, -0.004){20}{\circle*{0.02}}
%\multiput(-0.1, 0.6)(0.048, 0.033){13}{\circle*{0.02}}
\end{picture}

}

\end{picture}
\caption{Regular complex with branching into Whitney forms at index two.\newline
Gives a Stokes pair with discontinuous pressure.\label{fig:3dstokesl2}}
\end{figure}
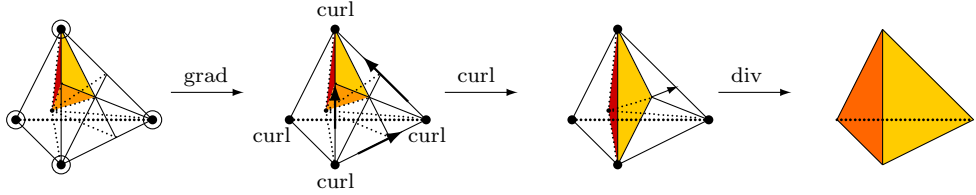
\end{center}
}

{
\begin{center}
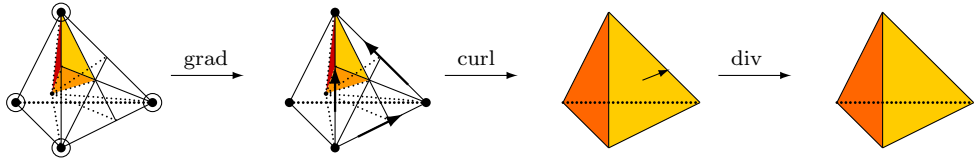
\begin{figure}
\setlength{\unitlength}{1.2cm}
\begin{picture}(2,2)(-0.5,0)
\put(0,0){
\begin{picture}(2,2)

{\color{red}\polygon*(-0.1, 0.6)(0, 0.9)(0, 1.5)
%{\color{lightorange}\polygon*(-0.1, 0.6)(0, 1.5)(0.368, 0.748)}
\color{lightorange}\polygon*(-0.1, 0.6)(0.368, 0.748)(0, 0.9)
\color{yellow}\polygon*(0, 0.9)(0, 1.5)(0.368, 0.748)
}
\put(-0.5, 0.5){\circle*{0.1}}
\put(-0.5, 0.5){\circle{0.2}}
\put(1, 0.5){\circle*{0.1}}
\put(1, 0.5){\circle{0.2}}
\put(0, 1.5){\circle*{0.1}}
\put(0, 1.5){\circle{0.2}}
\put(0, 0){\circle*{0.1}}
\put(0, 0){\circle{0.2}}
\put(0, 0){\line(-1, 1){0.5}}
\put(0, 0){\line(2, 1){1}}
\put(1, 0.5){\line(-1, 1){1}}
\put(-0.5, 0.5){\line(1, 2){0.5}}
\put(0, 0){\line(0, 1){1.5}}
\put(-0.1, 0.6){\circle*{0.05}}

\multiput(-0.5,0.5)(0.05,0){30}{\circle*{0.03}}

\put(0, 0){\line(1, 2){0.5}}
\put(0, 1.5){\line(1, -2){0.6}}
\put(1, 0.5){\line(-5, 2){1}}
\multiput(-0.1, 0.6)(0.05,0.015){10}{\circle*{0.02}}
\multiput(-0.1, 0.6)(0.05,-0.021){15}{\circle*{0.02}}
\multiput(-0.1, 0.6)(0.022,0.06){6}{\circle*{0.02}}
\multiput(-0.1, 0.6)(0.005, 0.05){18}{\circle*{0.02}}
\multiput(-0.1, 0.6)(0.008, -0.05){12}{\circle*{0.02}}
\multiput(-0.1, 0.6)(0.05, -0.004){20}{\circle*{0.02}}
\multiput(-0.1, 0.6)(0.048, 0.033){13}{\circle*{0.02}}
\end{picture}
}

\put(1.3, 0.8){\vector(1, 0){0.8}}
\put(1.44, 0.92){\footnotesize{$\grad$}}

\put(3,0){
\begin{picture}(2,2)

{\color{red}\polygon*(-0.1, 0.6)(0, 0.9)(0, 1.5)
%{\color{lightorange}\polygon*(-0.1, 0.6)(0, 1.5)(0.368, 0.748)}
\color{lightorange}\polygon*(-0.1, 0.6)(0.368, 0.748)(0, 0.9)
\color{yellow}\polygon*(0, 0.9)(0, 1.5)(0.368, 0.748)
}
\put(-0.5, 0.5){\circle*{0.1}}
%\put(-0.5, 0.5){\circle{0.2}}
\put(1, 0.5){\circle*{0.1}}
%\put(1, 0.5){\circle{0.2}}
\put(0, 1.5){\circle*{0.1}}
%\put(0, 1.5){\circle{0.2}}
\put(0, 0){\circle*{0.1}}
%\put(0, 0){\circle{0.2}}
\put(0, 0){\line(-1, 1){0.5}}
\put(0, 0){\line(2, 1){1}}
\put(1, 0.5){\line(-1, 1){1}}
\put(-0.5, 0.5){\line(1, 2){0.5}}
\put(0, 0){\line(0, 1){1.5}}
\put(-0.1, 0.6){\circle*{0.05}}

{
\linethickness{0.3mm}
\put(0,0.4){\vector(0, 1){0.5}}
\put(0.25, 0.125){\vector(2, 1){0.5}}
\put(0.8, 0.7){\vector(-1, 1){0.5}}
}

%\put(-0.2, 1.65){\footnotesize{$\curl$}}
%\put(-0.9, 0.25){\footnotesize{$\curl$}}
%\put(-0.2, -0.25){\footnotesize{$\curl$}}
%\put(0.8, 0.25){\footnotesize{$\curl$}}

\multiput(-0.5,0.5)(0.05,0){30}{\circle*{0.03}}

\put(0, 0){\line(1, 2){0.5}}
\put(0, 1.5){\line(1, -2){0.6}}
\put(1, 0.5){\line(-5, 2){1}}
\multiput(-0.1, 0.6)(0.05,0.015){10}{\circle*{0.02}}
\multiput(-0.1, 0.6)(0.05,-0.021){15}{\circle*{0.02}}
\multiput(-0.1, 0.6)(0.022,0.06){6}{\circle*{0.02}}
\multiput(-0.1, 0.6)(0.005, 0.05){18}{\circle*{0.02}}
\multiput(-0.1, 0.6)(0.008, -0.05){12}{\circle*{0.02}}
\multiput(-0.1, 0.6)(0.05, -0.004){20}{\circle*{0.02}}
\multiput(-0.1, 0.6)(0.048, 0.033){13}{\circle*{0.02}}
\end{picture}
}

\put(4.3, 0.8){\vector(1, 0){0.8}}
\put(4.44, 0.92){\footnotesize{$\curl$}}

\put(6,0){
\begin{picture}(2,2)

%{\color{deepcarrotorange}\polygon*(-0.5, 0.5)(0, 1.5)(1, 0.5)(0, 0)}
{\color{yellow}\polygon*(0, 1.5)(1, 0.5)(0, 0)
\color{darkorange}\polygon*(-0.5, 0.5)(0, 1.5)(0, 0)}

\put(0.368, 0.748){\vector(5, 2){0.3}}

%\put(-0.5, 0.5){\circle*{0.1}}
%\put(-0.5, 0.5){\circle{0.2}}
%\put(1, 0.5){\circle*{0.1}}
%\put(1, 0.5){\circle{0.2}}
%\put(0, 1.5){\circle*{0.1}}
%\put(0, 1.5){\circle{0.2}}
%\put(0, 0){\circle*{0.1}}
%\put(0, 0){\circle{0.2}}
\put(0, 0){\line(-1, 1){0.5}}
\put(0, 0){\line(2, 1){1}}

%\put(-0.2, 1.65){\footnotesize{$\div$}}
%\put(-0.9, 0.25){\footnotesize{$\div$}}
%\put(-0.2, -0.25){\footnotesize{$\div$}}
%\put(0.8, 0.25){\footnotesize{$\div$}}

\put(1, 0.5){\line(-1, 1){1}}
\put(-0.5, 0.5){\line(1, 2){0.5}}
\put(0, 0){\line(0, 1){1.5}}

\multiput(-0.5,0.5)(0.05,0){30}{\circle*{0.03}}

%\put(0, 0){\line(1, 2){0.38}}
%\put(0, 1.5){\line(1, -2){0.38}}
%\put(1, 0.5){\line(-5, 2){0.62}}
%\multiput(-0.1, 0.6)(0.05,0.015){10}{\circle*{0.02}}
%\multiput(-0.1, 0.6)(0.05,-0.021){15}{\circle*{0.02}}
%\multiput(-0.1, 0.6)(0.022,0.06){6}{\circle*{0.02}}
%\multiput(-0.1, 0.6)(0.005, 0.05){18}{\circle*{0.02}}
%\multiput(-0.1, 0.6)(0.008, -0.05){12}{\circle*{0.02}}
%\multiput(-0.1, 0.6)(0.05, -0.004){20}{\circle*{0.02}}
%\multiput(-0.1, 0.6)(0.048, 0.033){13}{\circle*{0.02}}
\end{picture}

}

\put(7.3, 0.8){\vector(1, 0){0.8}}
\put(7.44, 0.92){\footnotesize{$\div$}}

\put(9,0){
\begin{picture}(2,2)

%{\color{deepcarrotorange}\polygon*(-0.5, 0.5)(0, 1.5)(1, 0.5)(0, 0)}
{\color{yellow}\polygon*(0, 1.5)(1, 0.5)(0, 0)
\color{darkorange}\polygon*(-0.5, 0.5)(0, 1.5)(0, 0)}
%\put(-0.5, 0.5){\circle*{0.1}}
%\put(-0.5, 0.5){\circle{0.2}}
%\put(1, 0.5){\circle*{0.1}}
%\put(1, 0.5){\circle{0.2}}
%\put(0, 1.5){\circle*{0.1}}
%\put(0, 1.5){\circle{0.2}}
%\put(0, 0){\circle*{0.1}}
%\put(0, 0){\circle{0.2}}
\put(0, 0){\line(-1, 1){0.5}}
\put(0, 0){\line(2, 1){1}}
\put(1, 0.5){\line(-1, 1){1}}
\put(-0.5, 0.5){\line(1, 2){0.5}}
\put(0, 0){\line(0, 1){1.5}}
%\put(0.1, 0.6){\circle*{0.05}}

\multiput(-0.5,0.5)(0.05,0){30}{\circle*{0.03}}

%\multiput(-0.1, 0.6)(0.05,0.015){10}{\circle*{0.02}}
%\multiput(-0.1, 0.6)(0.05,-0.021){15}{\circle*{0.02}}
%\multiput(-0.1, 0.6)(0.022,0.06){6}{\circle*{0.02}}
%\multiput(-0.1, 0.6)(0.005, 0.05){18}{\circle*{0.02}}
%\multiput(-0.1, 0.6)(0.008, -0.05){12}{\circle*{0.02}}
%\multiput(-0.1, 0.6)(0.05, -0.004){20}{\circle*{0.02}}
%\multiput(-0.1, 0.6)(0.048, 0.033){13}{\circle*{0.02}}
\end{picture}

}

\end{picture}
\caption{Regular complex with branching into Whitney forms at index one.\label{fig:3dstokesl1}}
\end{figure}
\end{center}
}

For arbitrary $n$ and for $\ell = n-1$, which is perhaps the most interesting case from the point of view of Stokes equation, we are able to prove unisolvence:

\begin{proposition}\label{prop:last} Consider the spaces defined by (\ref{eq:adefkbis}) and the degrees of freedom given in particular by Proposition \ref{prop:dofell}, with $\ell = n-1$. The given degrees of freedom on $A^{n-1}(S)$ and $A^n(S)$ are unisolvent. The dimensions are $\dim A^{n-1}(S) = (n+1)^2$ and $\dim A^{n}(S) = 1$.

The associated interpolator commutes with the divergence operator. 

This gives a minimal good element for continuous vectorfields with discontinuous divergence.
\end{proposition}
\begin{proof}
\tpoint We have $\dim A^n(S) = 1$, since it consists of the constants.

\tpoint The proof of Proposition \ref{prop:nonegood} gives the lowerbound:
\begin{equation}
\dim A^{n-1}(S) \geq n(n+2) + 1 = (n+1)^2.
\end{equation}
which is the number of degrees of freedom defined in Proposition \ref{prop:dofell}.
\end{proof}

\begin{remark}\label{rem:gn14b}
In \cite{GuzNei14b} Stokes pairs (with discontinuous pressure) are defined in dimension 3. Among these, their so-called reduced element has the same degrees of freedom as the element we consider in Proposition \ref{prop:last}. Their vectorfields are defined using certain rational functions related to a 2D $\rmC^1$-element of Zienkiewicz. They also describe an element in arbitrary dimension with the same degrees of freedom as we have. In this generalization, the face-bubbles of Bernardi-Raugel \cite{BerRau85} are modified using the Bogovskii integral operator. The obtained vectorfields are therefore quite different from ours and perhaps less explicit.
\end{remark} 

\section*{Outlook}
We finish with some points that merit further investigation, and which we hope to address in a not too distant future:
\begin{itemize}
\item We have not included error estimates, but, given that we have defined natural degrees of freedom, we believe these could be obtained by combining techniques developed for HCT (e.g. \cite{Cia91} \S 46) with general techniques developed for FES (especially in \cite{ChrMunOwr11}).
 
\item A first natural extension of the present work, would be to define spaces with high approximation order in arbitrary dimension, in particular high order elements for Stokes in dimension 3.

\item It is also possible to use the framework of (generalized) FES to describe the complex consisting of the Morley element, the Crouzeix-Raviart element and the piecewise constants (see e.g. \cite{Bre15}). A general framework to discuss many existing non-conforming complexes is within reach.

\item The examples discussed in this paper consist of differential forms on domains in a vector space. It seems possible also to extend the techniques to manifolds. This would provide a new method, to solve say the shallow water equations on the sphere.

\end{itemize}

\section*{Acknowledgements}

We are grateful to Richard Falk for pointing out the paper \cite{ArnDouGup84}, which has interesting connections with this one. We are also grateful to Shangyou Zhang for numerous bibliographical remarks.

\medskip
 
\noindent SHC is supported by the European Research Council through the FP7-IDEAS-ERC Starting Grant scheme, project 278011 STUCCOFIELDS.

\medskip

\noindent KH is supported by the China Scholarship Council (CSC), project 201506010013 and by the European Research Council through the FP7-IDEAS-ERC Advanced Grant scheme, project 650138 FEEC-A. The stimulating collaborations are achieved during his visit at University of Oslo (UiO) since September 2015.  He is grateful for the kind hospitality and support of UiO.

\bibliography{../Bibliography/alexandria,../Bibliography/newalexandria,../Bibliography/mybibliography}{}
\bibliographystyle{plain}
\end{document}